\documentclass[11pt]{amsart} 
\usepackage{amsmath, amsfonts}
\usepackage{amsthm}
\usepackage{amssymb}
\usepackage[centering, vcentering, marginratio=1:1, vscale=0.8, hscale=0.7, letterpaper]{geometry}
\usepackage{url} 
\usepackage[colorlinks,  citecolor=black, linkcolor = black, urlcolor = black]{hyperref} 
\usepackage{mathtools}
\usepackage{tikz-cd}
\usepackage{stmaryrd}
\usepackage{mathrsfs}
\usepackage{derivative}
\usepackage{enumerate}

\theoremstyle{plain}
\newtheorem{thm}{Theorem}[section]
\newtheorem{prop}[thm]{Proposition}
\newtheorem{cor}[thm]{Corollary}
\newtheorem{lem}[thm]{Lemma}

\theoremstyle{definition}
\newtheorem{defn}[thm]{Definition}

\newtheorem{defn-prop}[thm]{Definition-Proposition}

\theoremstyle{remark}
\newtheorem{remark}[thm]{Remark}

\DeclareMathOperator{\Hom}{Hom}

\DeclareMathOperator{\Spec}{Spec}

\DeclareMathOperator{\Proj}{Proj}

\DeclareMathOperator{\an}{an}
\DeclareMathOperator{\ma}{MA}

\DeclareMathOperator{\vol}{vol}

\DeclareMathOperator{\Val}{Val}

\DeclareMathOperator{\triv}{triv}

\DeclareMathOperator{\qm}{qm}

\DeclareMathOperator{\fs}{\mathrm{FS}}
\DeclareMathOperator{\cpsh}{\mathrm{CPSH}}
\DeclareMathOperator{\dcpsh}{\mathrm{d-CPSH}}
\DeclareMathOperator{\psh}{\mathrm{PSH}}

\DeclareMathOperator{\I}{I}
\DeclareMathOperator{\J}{J}
\DeclareMathOperator{\E}{E}

\newcommand{\GG}{\mathbb{G}}

\newcommand{\NN}{\mathbb{N}}

\newcommand{\QQ}{\mathbb{Q}}
\newcommand{\RR}{\mathbb{R}}

\newcommand{\TT}{\mathbb{T}}

\newcommand{\ZZ}{\mathbb{Z}}

\newcommand{\kk}{\Bbbk}

\newcommand{\cE}{\mathcal{E}}
\newcommand{\cF}{\mathcal{F}}

\newcommand{\cH}{\mathcal{H}}

\newcommand{\cM}{\mathcal{M}}

\newcommand{\fm}{\mathfrak{m}}

\newcommand{\ft}{\mathfrak{t}}

\newcommand{\la}{\langle}
\newcommand{\ra}{\rangle}

\title{Some Non-Archimedean pluripotential theory on polarized affine cones}
\author{Yueqiao Wu}
\address{
        School of Mathematics\\
        Institute for Advanced Study\\
        Princeton, NJ 08540, USA
	}
\email{yueqiaow@ias.edu}

\begin{document}
\begin{abstract}
    We undertake a preliminary step towards studying non-Archimedean pluripotential theory on polarized affine cones over a trivially valued field. We study plurisubharmonic functions and the Monge--Amp\`ere operator defined on the finite energy class, partially generalizing a result of Boucksom--Jonsson on projective varieties. 
\end{abstract}
\maketitle
\setcounter{tocdepth}{1}
\tableofcontents

\section{Introduction}
Throughout the paper, we work over a trivially valued algebraically closed field of characteristic zero. The goal of this paper is to serve as a starting point to understand pluripotential theory on Berkovich analytifications of polarized affine cones initiated in~\cite{wu}, following the framework of Boucksom--Jonsson~\cite{BJ18}. While we do not generalize all of the global theory in \emph{loc.cit.} for now, we aim to present what can be carried out in our setting, describe the key obstacles, and provide sufficient background for an application to local K-stability.

Complex pluripotential theory has played an important role towards solving the complex Monge--Amp\`ere equation via a variational approach in~\cite{BBGZ}. Inspired by this, the variational approach to the Yau--Tian--Donaldson (YTD) conjecture for log Fano varieties (see~\cite{BBJ, Li19}) motivates a systematic study of global non-Archimedean pluripotential theory over trivially valued fields. As developed by Boucksom--Jonsson in~\cite{BJ18}, the global theory provides a variational approach to the solution of a non-Archimedean Monge--Amp\`ere equation over trivially valued fields. This is analogous to a sequence of celebrated results on Monge--Amp\`ere equations in different settings over the past few decades, dating back to Yau's theorem~\cite{Yau78} solving the Calabi conjecture; and more recently a non-Archimedean version of that over discretely valued fields proved in~\cite{BFJ15}.

In this paper, we focus on studying the non-Archimedean pluripotential theory on affine cones over projective varieties. We briefly explain the local setting that we are working in. Let $(X = \Spec R, \TT, \xi)$ be a polarized affine cone, i.e. $X$ is a normal affine variety, $\TT$ is a torus of automorphisms with a unique fixed point $o\in X$ with ideal $\fm$, and $\xi$ is a vector generating a subtorus of $\TT$ and acting with positive weights. Each $\xi$ gives a $\TT$-invariant quasi-monomial valuation $v_\xi$ on $X$ centered at $o$, and it is divisorial if and only if $\xi$ generates a $\GG_m$ action. In the case when $v_\xi$ is divisorial, one can realize $X$ as the affine cone over some polarized pair. As originally introduced in~\cite{CSIrregular}, building on the ideas of~\cite{MSY}, 
there is a local K-stability theory for polarized affine cones, which agrees with the one for polarized pairs when $v_\xi$ is divisorial. Moreover, it is a continuous extension of the global K-stability theory in the sense that the Futaki invariant in \emph{loc.cit.} depends continuously on $\xi$. Relevant YTD type results have been obtained for log Fano cone singularities, that is, $\QQ$-Gorenstein polarized affine cones with klt singularities, see~\cite{CS19, LX18, LWX, Li21}. This also motivates the study of NA pluripotential theory on polarized affine cones.

We now briefly explain the key ideas and difficulties in the local case.
In the global setting, the theory of Boucksom--Jonsson relies on the intersection theory and positivity results provided by the geometry of projective varieties. From this point of view, one might hope the existing local intersection theory, for example Hilbert--Samuel multiplicities, would lead to a local pluripotential theory in our setting. This is in fact the direction that has been extensively studied in~\cite{BLQ}, but will not be the point of view taken here. Philosophically, the local setting we are working in is different from the traditional one in algebraic geometry and commutative algebra, in the sense that the Reeb field provides a polarization in addition to the information on the singularity of the cone point. Nevertheless, the theory by Chambert--Loir and Ducros (see~\cite{CLD12}) we are relying on has a more analytic and local nature. Thus one would hope that a similar local non-Archimedean pluripotential theory can be developed, paralleling the global story. 

\subsection*{Main result}
Let $(X=\Spec R = \Spec \bigoplus_\alpha R_\alpha, \xi)$ be a polarized affine cone described before, and write $X^{\an}$ for the Berkovich analytification of $X$ with respect to the trivial valuation on $\kk$. This consists of all semivaluations on $X$ extending the trivial valuation on $\kk$, equipped with the topology of pointwise convergence. Motivated by Sasakian geometry, the \emph{NA link} is defined to be 
\[X_0\coloneqq \{v\in X^{\an}\mid v(\fm)=0\}.\] 
This is a compact analytic domain in $X^{\an}$, independent of the polarization $\xi$.  While we mostly work with punctured cone $X^{\an}\setminus\{o\}$ when defining different classes of functions, the cone structure on $X$ allows us to think of them as functions on $X_0$, via certain $\xi$-equivariance. More importantly, the Monge--Amp\`ere measures will always be supported on $X_0$, and it will be convenient to deal with convergence theory on $X_0$.

Analogous to the Fubini--Study metrics studied in~\cite{BJ18}, the class $\cH(\xi)$ of \emph{Fubini--Study functions} was introduced in~\cite{wu} as continuous functions $\varphi: X^{\an}\setminus\{o\}\to \RR$ of the form
\[
\varphi = \max_{\alpha_j}\{\frac{\log|f_{\alpha_j}|+\lambda_j}{\la\xi, \alpha_j\ra}\}
\]
for a finite set of $f_{\alpha_j}\in R_{\alpha_j}$ generating an $\fm$-primary ideal. In particular, we always have a canonical reference function $\varphi_\xi$, where all $\lambda_j=0$ and $f_{\alpha_j}$'s generate the maximal ideal $\fm$. One can alternatively identify $X_0$ with the set $\{\varphi_\xi=0\}$, analogous to the construction in Sasakian geometry, see e.g.~\cite{sasaki}. Unlike FS metrics, FS functions behave differently when $\xi$ is irrational. For example, adding two FS functions does not give a FS function in general. A more subtle difference is that the Monge--Amp\`ere measures will have support on quasi-monomial valuations 
instead of just divisorial valuations. To get around this issue, for a fixed polarization $\xi$, we are going to enlarge the space of FS functions first by passing to its closure in the space of $\xi$-equivariant continuous functions on $X^{\an}$, which we call the space of \emph{continuous psh functions} $\cpsh(\xi)$. Following the same pattern as in the complex analytic and global non-Archimedean theory, we define a psh function to be an upper-semicontinuous (usc) function that is the pointwise limit of a decreasing net of continuous psh functions,
the Monge--Amp\`ere energy defined in~\cite{wu} naturally extends to a functional $\E: \psh(\xi)\to \RR\cup \{-\infty\}$ on psh functions. The finite energy class, which in particular contains $\cpsh(\xi)$, is defined to be
\[\cE^1(\xi)\coloneqq \{\varphi\in \psh(\xi): \E(\varphi)>-\infty\}.\]
One can dually define the energy $\E^\vee: \cM(X_0)\to [0,+\infty]$ from the space of Radon probability measures on $X$ supported on $X_0$ via
\[\E^\vee(\mu; \xi) = \sup_{\varphi\in \cE^1(\xi)}\left(\E(\varphi)-\int \varphi \mu\right),\]
and the space of measures of finite energy by 
\[\cM^1(\xi)\coloneqq \{\mu\in \cM(X_0): \E^{\vee}(\mu)<+\infty\}.\]
There are also natural quasi-pseudometrics $\I$ on $\cE^1(\xi)$ and dually $\I^\vee$ on $\cM^1(\xi)$.
We note that by construction one has $\ma(\cH(\xi))\subset \cM^1(\xi)$. As a local analog of~\cite[Theorem 7.14]{BJ18}, our main result extends the Monge--Amp\`ere operator from $\cH(\xi)$ to $\cE^1(\xi)$:
\begin{thm}\label{main-thm-pluripotential}
There is a unique extension of the Monge--Amp\`ere operator $\ma: \cH(\xi)\to \cM^1(\xi)$ to 
$\ma: (\cE^1(\xi), \I)\to (\cM^1(\xi), \I^\vee)$, which is continuous along decreasing nets in $\cE^1(\xi)$, and bi-Lipschitz with respect to $\I$ and $\I^\vee$.
\end{thm}

\subsection*{Outlook and relation to other works}
A key missing ingredient that prevents us from generalizing Boucksom--Jonsson's result on studying the image of the Monge--Amp\`ere operator in $\cM^1$ in full generality is the uniform differentiability result, see~\cite[Theorem 8.5]{BJ18}. This ultimately boils down to understanding the relation between function classes with different polarizations. For example, it will be helpful to understand how $\psh(\xi)$ and $\psh(\xi')$ relate when $\xi, \xi'$ are not multiples of each other.
We remark that this is a triviality in the global case, due to the fact that one can add two FS functions with different polarizations. There is also an analogous statement in the complex analytic setting, see~\cite[Lemma 6.1]{HeLi}. 

The next step is to try to solve a local NAMA equation. We expect that the general framework of Boucksom--Jonsson works in this case as well, but one may need to deal with the subtlety that on a irrationally polarized affine cone, the Monge--Amp\`ere operator sends a divisorial norm not to a divisorial measure, but to a quasi-monomial measure, i.e. a measure supported at genuinely quasimonimial but not divisorial points. This is in fact one of the main difficulties in this paper already. 

It is also interesting to understand how much of the local theory would fit into the synthetic approach proposed in~\cite{BJNASyn}. While our setting is slightly different from theirs, and this is not the approach we are taking here, the same strategy can be used to define the space of measures of finite energy and form the duality without introducing functions of finite energy.

We also mention briefly an application of the pluripotential theory in our forthcoming preprint. Specifically, the Monge--Amp\`ere operator defined on the finite energy class gives an alternative characterization of the Monge--Amp\`ere energy using measures of finite energy, see Proposition~\ref{chap4-prop-E-duality}. Using the duality of energy of psh functions and measures, and following the ideas of~\cite{BBJ}, one can show that to test K-semistability of log Fano cone singularities, it suffices to test special test configurations. This is a local analog of~\cite{LX14}.

\subsection*{Structure of the paper} In Section~\ref{section: prelim}, we recall basic definitions and notation, and results in~\cite{wu} on the Monge--Amp\`ere operator defined on test configurations, or equivalently the space of FS functions. We extend the theory to continuous psh functions in Section~\ref{section: cpsh} and Section~\ref{section: EnergyOnPSH}. Section~\ref{section: generalPSH} is devoted to the theory of general psh functions. The first part of the main theorem is proved in Section~\ref{section: ExtendMA}. Finally in Section~\ref{section: FiniteEnergyMeasure}, we introduce the space of measures of finite energy and complete the proof of the main theorem.

\subsection*{Acknowledgement}
This paper grew out of a part of author's thesis. She would like to thank her advisor, Mattias Jonsson, for helpful discussions. This work is partially supported by NSF grants DMS-1900025, DMS-2154380 and DMS-1926686.

\subsection*{Notation and conventions}
For $x, y\in \RR_+$, we write $x\lesssim y$ when $x\leq C(n)y$ for some dimensional constant $C(n)$, and $x\approx y$ if $x\lesssim y$ and $y\lesssim x$. By a quasi-pseudometric on a set $Z$, we mean a pseudometric $d$ on $Z$ satisfying the quasi-triangle inequality $d(x, y)\leq C(d(x, z)+d(z, y))$ for some constant $C>0$.

\section{Preliminaries}\label{section: prelim}
\subsection{Polarized affine cones}
Let $X=\Spec R$ be a normal affine variety of dimension $n$ over a field $\kk$, and let $\TT = \GG_m^r$ be an algebraic torus acting on $X$. Such $X$ is called a \emph{$\TT$-variety} if the $\TT$-action on $X$ is \emph{effective}, i.e., no elements of the torus acts trivially on all of $X$ other than the identity. For our purposes, we will also assume that the $\TT$-action is \emph{good}, that is, it in addition has a unique fixed point, which is contained in the closures of all orbits. We refer to~\cite{AIPSV, PS, AH06, AHS08} for more a detailed study on $\TT$-varieties.

We write $M\coloneqq \Hom(\TT, \GG_m)$ for the weight lattice, and $N\coloneqq M^\vee = \Hom(\GG_m, \TT)$ the dual lattice. Then we have a weight decomposition $R=\bigoplus_{\alpha}R_\alpha$.
Put $N_\RR = N\otimes_\ZZ \RR$, and $\Lambda\coloneqq \{\alpha: R_\alpha \neq 0\}$. We will denote by $o$ the cone point, and $\fm$ the maximal ideal defining the cone point.
\begin{defn}
    The \emph{Reeb cone} is 
    \[\ft_\RR^+\coloneqq \{\xi\in N_\RR: \la \xi, \alpha\ra >0, \ \forall \alpha\in \Lambda\setminus\{0\}\}.\]
    A vector $\xi\in \ft_\RR^+$ is called a \emph{Reeb field}.
    A \emph{polarized affine cone} is a triple $(X, \TT; \xi)$ with $X$ a normal affine $\TT$-variety with good torus action and $\xi$ a Reeb field. We will often omit $\TT$ and write $(X; \xi)$ if the torus action is clear from the context.
\end{defn}

\begin{remark}
    We remark the choice of terminology here. Our definition agrees with the one for polarized affine varieties, as originally introduced in~\cite{CSIrregular}. If one imposes further conditions on $X$ to make it $\QQ$-Gorenstein and have klt singularities, then it is usually called log Fano cone singularities in the literature. The change of terminology is to avoid a similar, but different notion of a polarized affine variety considered in~\cite{Sun23}. 
\end{remark}

\begin{defn}
    A Reeb field $\xi$ is called \emph{quasi-regular} or \emph{rational} if $\xi\in \ft_\QQ^+\coloneqq \ft_\RR^+\cap N_\QQ$. It is called \emph{irregular} or \emph{irrational} otherwise. 
    For a quasi-regular Reeb field $\xi$, we define its \emph{primitive vector} $\hat \xi\coloneqq l\xi$ as the multiple of $\xi$ such that $l$ is smallest with $\la l\xi, \alpha\ra\in \ZZ, \forall \alpha\in \Lambda$. 
\end{defn}

When $\xi$ is rational, the torus $\TT(\xi)$ it generates is simply $\GG_m$, in which case we can recover $X$ as an affine cone over a polarized orbifold pair, i.e. a projective orbifold with an ample orbiline bundle. 

\begin{prop}\label{prelim-prop-quotient-log-fano}
    Let $(X; \xi)$ be a polarized affine cone with $\xi\in \ft_\QQ^+$. Then $X$ is the affine cone over a polarized pair $(V, B_V; L)$.
\end{prop}
\begin{proof}
    Assume $\xi\in \frac 1l \NN$. Since $\TT(\xi)=\GG_m$, one can think of $X\setminus\{o\}$ as the complement of the zero section in the total space of an ample orbiline bundle over some orbifold (see~\cite{RossThomas}). In algebro-geometric terms, $X\setminus \{o\}$ is a Seifert $\GG_m$-bundle over a projective variety $V = \Proj R$. It is shown in~\cite{kollar} that the coordinate ring $R$ with the new grading induced by the $\xi$ action can be viewed as the section ring of some $\QQ$-line bundle $L$ on $V$:
    \[R = \bigoplus_\alpha R_\alpha = \bigoplus_k\left(\bigoplus_{\la \xi, \alpha\ra = \frac kl}R_\alpha\right) = \bigoplus_{k\in \ZZ} H^0(V, kL).\]
    If the fractional part of $L$ is given by $\sum_i \frac{a_i}{b_i}D_i$, we put $B_V\coloneqq \sum_i(1-\frac{1}{b_i})D_i$ as the orbifold boundary of $V$ to remember the orbifold structure.
\end{proof}

\subsection{Valuations and valuative invariants}
Let $X$ be a variety of dimension $n$. We first recall the space of valuations explored in~\cite{JM12}.
\begin{defn}
	A (real) \emph{semivaluation} $v$ on $X$ is a map $v: K(X)\to \RR \cup \{\infty\}$ such that 
	\begin{enumerate}
		\item $v(f+g)\ge \min\{v(f), v(g)\} ,  \forall f, g\in R$;
		\item $v(fg) = v(f)+v(g), \forall f, g\in R$;
		\item $v(0)=\infty;$
		\item $v|_{\kk^*} = 0$.
	\end{enumerate}
	A \emph{valuation} $v$ is a semivaluation such that $v(f)=\infty$ if and only if $f=0.$
\end{defn}

The Berkovich analytification functor associates to each variety a good $\kk$-analytic space, which can be thought of as a natural compactification of the space of valuations.
For us, we will only consider the case when $X=\Spec R$ is an affine variety over a trivially valued field $\kk$. Then as a set, 
\[X^{\an}  = \{\mathrm{multiplicative \ semivaluations \ on \ } R \mathrm{ \ that \ are \ trivial \ on \ } \kk\}.\]
The topology on $X^{\an}$ is the weakest one such that $v\mapsto v(f)$ is continuous for all $f\in R.$ As a topological space, $X^{\an}$ is locally compact and  Hausdorff.
The space $X^{\an}$ contains the set of valuations as a dense subset. 
For a polarized affine cone $(X; \xi)$, the torus action also induces a map
\[|\TT^{\an}|\times |X^{\an}|\to |X^{\an}|, (\xi, v)\mapsto \xi*v,\]
where $\xi$ is identified with a vector in $N_\RR\subset T^{\an}$, and for $f=\sum f_\alpha\in R$,
\[(\xi*v)(f) = \min_\alpha\{v(f_\alpha)+\la \xi, \alpha\ra\}.\]
Recall from~\cite[Lemma 2.8]{wu} that $v\in X^{\an}$ is $\TT$-invariant iff for any $f=\sum_\alpha f_\alpha$, $v(f)=\min_\alpha v(f_\alpha)$.
This leads us to the following definition.
\begin{defn}
    Let $\fm$ be the maximal ideal defining the unique fixed point $o$ by the torus action. The \emph{non-Archimedean (NA) link} of $X$ at $o$ is 
    \[X_0\coloneqq \{v\in X^{\an} : v(\fm)=0\}.\]
    We will denote by $X_0^{\TT}$ the set of $\TT$-invariant points in $X_0$.
\end{defn}
\begin{remark}
We remark that in the literature, the NA link usually refers to $\Val_{X, o}$ where $o$ is the cone point. Our choice of this name is related to the notion of a link in Sasaki geometry.
\end{remark}

Similar to~\cite{BlumJonsson}, any $\TT$-invariant valuation $v$ centered on $X$ defines a $\TT$-invariant filtration via
\[
\cF_v^\lambda R_\alpha\coloneqq \{f_\alpha\in R_\alpha\mid v(f_\alpha)>\lambda\}.
\]
\begin{defn-prop}\label{chap3-prop-T-formula}
    A valuation $v\in X_0$ is \emph{linearly bounded} if 
    \[T(v;\xi)\coloneqq \sup \{\frac{v(f_\alpha)}{\la \xi, \alpha\ra}: f_\alpha\in R_\alpha, \alpha\in \Lambda\setminus\{0\}\}<\infty.\]
    For any linearly bounded $v\in X_0$, $T(v; \xi)$ is continuous with respect to $\xi$, and homogeneous of degree $-1$. Further, $T(v;\xi)=T(\cF_v;\xi)$ as defined in~\cite{wu}.
\end{defn-prop}
\begin{proof}
    The same proof as in~\cite[Proposition 3.15]{wu} shows the first assertion. For the second assertion, note that the equality is true for $\xi$ primitive. In general, it follows from homogeneity and continuity in $\xi$.
\end{proof}

Recall the \emph{volume} of a linearly bounded filtration is defined in~\cite{wu}, and here we are going to look at only filtrations associated to linearly bounded valuations, see also~\cite{XuZhuang20}. In this context, we write
\[
S(v;\xi)\coloneqq S(\cF_v; \xi) = \lim_{m\to \infty}\frac{\sum a_{m, j} }{m\dim R_m},
\]
where $R_m=\bigoplus_{\alpha\colon \la \xi, \alpha\ra\leq m} R_\alpha$, and $\{a_{m, j}\}$ are the jumping numbers on $R_m$ associated to the filtration $\cF_v$.
This is continuous with respect to $\xi$, and homogeneous of degree $-1$.

\subsection{Fubini--Study functions and the Monge--Amp\`ere operator}
We now recall Fubini--Study functions and the Monge--Amp\`ere operator defined in~\cite{wu} along with some of their properties. 
\begin{defn}
	A Fubini--Study function on $X^{\an}$ with polarization $\xi$ is a function of the form
	
	\[\varphi = \max_{1\leq j\le N} \{\frac{\log|f_j|+\lambda_j}{\la \xi, \alpha_j\ra}:   f_j\in R_{\alpha_j}\}\]
	where $\bigcap_{j=1}^N \{f_j = 0\} = \{o\}$, $\log|f_j|(v)\coloneqq -v(f_j)$, and $\lambda_j\in \RR$.
We will denote by $\cH(\xi)$ the set of Fubini--Study functions on $X^{\an}$ with polarization $\xi$, and write $\varphi_\xi\coloneqq \max\{\frac{\log|f_\alpha|}{\la\xi,\alpha\ra}\}$.
\end{defn}

The following lemma follows directly from the definition. 
\begin{lem}\label{chap4-lem-fs-properties}
    With notation as above, we have
    \begin{enumerate}
        \item If $\varphi\in \cH(\xi)$, then $\varphi+c\in \cH(\xi), \forall c\in \RR$;
        \item If $\varphi_1, \varphi_2\in \cH(\xi)$, then $\max\{\varphi_1, \varphi_2\}\in \cH(\xi)$;
        \item If $\varphi\in \cH(\xi)$, then $a\varphi \in \cH(\frac 1a \xi)$ for $a>0$.
        \item If $\xi$ is rational, and $\varphi_1\in \cH(a\xi), \varphi_2\in \cH(b\xi)$, then $\varphi_1+\varphi_2\in \cH\left((\frac 1a+\frac 1b)^{-1}\xi\right)= \cH(\frac{ab}{a+b}\xi) $.
    \end{enumerate}
\end{lem}
Given a FS function $\varphi\in \cH(\xi)$, if $\xi_i\to \xi$, then we get a sequence of $\varphi^i\in \cH(\xi_i)$ by only changing the denominator from $\xi$ to $\xi_i$.
\begin{lem}\label{lem-UnifConv}
    With the notation above, we have that $\varphi^i$ converges to $\varphi$ locally uniformly on $X^{\an}\setminus\{0\}.$
\end{lem}

\begin{proof}
    Fix a compact subset $K\subset X^{\an}\setminus\{0\}$. Without loss of generality, we can assume all $\log |f_j|$ are finite valued on $K$. Indeed, if $\log|f_j| $ can take $-\infty$ on $K$ for some $j$, then we can write $K = K_1\cup K_2$ where $K_1= K\cap \{v(f_j)\geq -M\} $ and $K_2 = K\cap \{v(f_j)\leq -M\}$ for some $M\ll \min_K\varphi $. On $K_1$ all $\log|f_j|$ are finite valued, and on $K_2$ we can drop the term involving $\log|f_j|$ in our expression of $\varphi$. 

    Now it suffices to show that if $\log |f|$ is finite valued on $K$, then $\psi_i = \frac{\log|f|+\lambda}{\langle\xi_i,\alpha\rangle}$ converges to $\psi =\frac{\log|f|+\lambda}{\langle\xi,\alpha\rangle} $ uniformly on $K$. This is because 
    \[|\psi_i-\psi| \leq \sup_K(|\log|f|+\lambda|) |\frac{1}{\langle\xi_i,\alpha\rangle} - \frac{1}{\langle\xi,\alpha\rangle}|\to 0\]
    uniformly.
    \end{proof}

As observed in~\cite{wu}, the class of Fubini--Study functions are in particular psh-approachable in the sense of Chambert--Loir and Ducros, whose theory allows us to build the NA Monge--Amp\`ere measure:
\begin{thm}[{\cite[Corollary 5.7]{wu}}]
    For $\varphi_1, \cdots, \varphi_{n-1}\in \cH(\xi)$, the associated Monge--Amp\`ere measure
    \[\ma(\varphi_1, \cdots, \varphi_{n-1};\xi)\coloneqq \frac{1}{\vol(\xi)}d'd''\varphi_1\wedge \cdots \wedge d'd''\varphi_{n-1}\wedge d'd''\varphi_\xi^+,\]
    where $\varphi_\xi^+\coloneqq \max\{\varphi_\xi, 0\}$, 
    is a probability measure supported on $X_0^{\TT}$.
\end{thm}
When $\varphi=\varphi_1=\cdots =\varphi_{n-1}$, and $\xi$ is clear from context, we will simply write $\ma(\varphi)$ for simplicity. We remark that the above assignment is also symmetric and multilinear. 
    
\section{Continuous plurisubharmonic functions}\label{section: cpsh}
In this section we introduce a broader class of plurisubharmonic (psh) functions than FS functions. Unlike the global case, where the sum of two FS functions is still a FS function; more precisely, one can add two metrics on two line bundles $L_1, L_2$ respectively to get a new metric on $L_1+L_2$, FS functions are more rigid in the local case. First, it is not clear what it means to add two FS functions with non-proportional polarizations. While we believe there is some intepretation that can be made in this situation, this shall not be our focus here. Thus in the sequel, we will often fix a polarization $\xi$. 
Second, when $\xi$ is irrational, the sum of two FS functions with the same polarization may not be a FS function anymore. Thus we need to start with a larger class: the class of continuous psh functions. 
\begin{defn}
The space of \emph{continuous $\xi$-psh functions} $\cpsh(\xi)$ is the closure of $\cH(\xi)$ in the space $C^0(X^{\an}, \xi)$ of continuous $\xi$-equivariant functions with topology given by local uniform convergence.
\end{defn}

It follows directly from the definition that this class of functions are indeed continuous and psh-approachable in the sense of Chambert--Loir and Ducros. The following lemma gives a more concrete description of cpsh functions.
\begin{lem}
    Given any $\varphi\in \cpsh(\xi)$ and $\xi_i$ converging to $\xi$, there is a sequence $\varphi_i\in\cH(\xi_i)$ converging locally uniformly to $\varphi$.
\end{lem}
\begin{proof}
Let $\varphi\in \cpsh(\xi)$, and assume
\[\varphi_i' = \max_{1\leq j\leq N_i}\{\frac{\log|f_{ij}|+\lambda_{ij}}{\langle\xi, \alpha_{ij}\rangle}: f_{ij}\in R_{\alpha_{ij}}\}\in \cH(\xi)\]
converges locally uniformly to $\varphi$. Let $\xi_i\to \xi$ be a sequence of Reeb fields converging to $\xi$. Define
\[\varphi_i = \max_{1\leq j\leq N_i}\{\frac{\log|f_{ij}|+\lambda_{ij}}{\langle\xi_i, \alpha_{ij}\rangle}: f_{ij}\in R_{\alpha_{ij}}\}\in \cH(\xi_i).\]
We claim that $\varphi_i$ also converges to $\varphi$ locally uniformly. Observe that on a fixed compact set $K$, there is some $M>0$ such that $|\varphi_i|\leq M$, and as in Lemma~\ref{lem-UnifConv}, after assuming all $\log|f_{ij}|$'s are finite-valued, we have
\begin{align*}
    |\varphi_i'-\varphi_i|\leq \max\left\{\left|\frac{\log|f_{ij}|+\lambda_{ij}}{\langle\xi_i, \alpha_{ij}\rangle}\frac{\la \xi-\xi_i, \alpha_{ij}\rangle}{\langle \xi, \alpha_{ij}\rangle}\right|\right\}\leq M\max\left\{\left|\frac{\la \xi-\xi_i, \alpha_{ij}\rangle}{\langle \xi, \alpha_{ij}\rangle}\right|\right\}\to 0,
\end{align*}
where the last convergence is uniform in $\alpha$.
Hence $|\varphi_i-\varphi|\leq |\varphi_i-\varphi_i'|+|\varphi_i'-\varphi|\to 0$ locally uniformly as $i\to \infty$.
\end{proof}
In view of the lemma, we sometimes identify a continuous $\xi$-psh function with its restriction as a function on $X_0$.
With this larger class of functions, we have the following properties of continuous psh functions, similar to the that of global FS metrics (see e.g.~\cite[Lemma 2.4]{BJ18v1}, \cite[Proposition 3.6]{BJ18}).
\begin{prop}\label{chap4-prop-cpsh-vector-space}
Let $(X^{\mathrm{an}}; \xi)$ be a polarized affine variety. Then
\begin{enumerate}[(1)]
    \item If $\varphi\in \cpsh(\xi)$, then $\varphi+c\in \cpsh(\xi)$ for all $c\in \mathbb{R}$.
    \item If $\varphi\in \cpsh(a\xi), \psi\in \cpsh(b\xi)$ for some $a, b\in \mathbb{R}_{>0}$, then $\varphi+\psi \in \cpsh(\frac{ab}{a+b}\xi)$.
    \item If $\varphi, \psi \in \cpsh(\xi),$ then $\max\{\varphi, \psi\}\in \cpsh(\xi)$.
    \item If $\varphi \in \cpsh(\xi)$, then $a\varphi \in \cpsh(\frac 1a \xi)$ for all $a\in \mathbb{R}_{>0}$.
    \item $\cpsh(\xi)$ is a convex set, i.e. if $\theta_1, \theta_2\in [0,1]$ such that $\theta_1+\theta_2 = 1$, and $\varphi, \psi\in \cpsh(\xi)$, then $\theta_1\varphi+\theta_2\psi \in \cpsh(\xi)$.
\end{enumerate}
\end{prop}

\begin{proof}
(1) and (4) are immediate consequences of the corresponding properties of FS functions in Lemma~\ref{chap4-lem-fs-properties}. (5) follows from (2) and (4).

In view of the previous lemma, we can now assume there are sequences $\varphi_i\in\fs(a\xi_i)$, $\psi_i\in \fs(b\xi_i)$ approaching $\varphi,\psi$ with $\xi_i$ rational. Then by Lemma~\ref{chap4-lem-fs-properties}, $\varphi_i+\psi_i\in \fs(\frac{ab}{a+b}\xi_i)$, and $\varphi_i+\psi_i$ converges to $\varphi+\psi$ locally uniformly. This proves (2).

For (3), we pick $\varphi_i, \psi_i\in \fs(\xi_i)$ with $\xi_i$ rational and $\xi_i\to \xi$ approximating $\varphi$ and $\psi$ respectively. Then on any compact subset, one has
\begin{align*}
    \left|\max\{\varphi_i, \psi_i\} - \max\{\varphi,\psi\}\right|
    &=\frac 12 \left|\varphi_i+\psi_i-\varphi-\psi+|\varphi_i-\psi_i|-|\varphi-\psi|\right|\\
    &\leq \frac 12\left(|\varphi_i-\varphi|+|\psi_i-\psi|+|\varphi_i-\psi_i-\varphi+\psi|\right)\\
    &\to 0 \mathrm{, as \ } i\to \infty.
\end{align*}
\end{proof}

In what below, we define a local analogue of ``the difference of FS metrics" studied in~\cite{BJ18v1} (known as PL functions in~\cite{BJ18}). Again, we can either think of them as functions on $X^{\an}\setminus\{o\}$ or $X_0$ depending on their $\xi$-equivariance properties. These functions in particular form a dense subset of continuous functions on the NA link $X_0$.
\begin{defn}
We say a function $f$ on $X^{\an}$ is a \emph{difference of continuous psh functions with polarization $\xi$} if $f=\varphi-\psi$ for some $\varphi\in \cpsh(a\xi), \psi\in \cpsh(b\xi)$, where $a=\frac{b}{1+b}\in \mathbb{R}_{>0}$. The set of difference of bounded psh functions on $X^{\an}\setminus \{o\}$ with polarization $\xi$ is denoted by $\dcpsh(X,\xi)$. 
We denote by $\dcpsh_\xi(X)$ the set of functions of the form $\varphi-\psi$ where $\varphi,\psi \in \cpsh(a\xi)$ for some $a\in\mathbb{R}_{>0}$.
\end{defn}
We note that the former is a class of $\xi$-equivariant functions, and the latter class is $\xi$-invariant. It follows easily from the definition that if $\varphi\in \dcpsh(X, \xi)$, and $u\in \dcpsh_\xi(X)$, then $\varphi+u\in \dcpsh(X, \xi)$.
\begin{thm}
The space $\dcpsh_\xi(X)$ is an $\RR$-vector space and is a dense subset of $C^0(X_0)$.
\end{thm}

\begin{proof} 
We follow the proof as in~\cite[Theorem 2.7]{BJ18v1} and~\cite[Theorem 7.12]{Gubler}. By the lattice version of Stone--Weierstrass Theorem, it suffices to show that $\dcpsh_\xi(X)$ is a $\mathbb{Q}$-vector space which is stable under max and contains all constants, and that it also separates points in $X_0$. The same proof shows that $\dcpsh_\xi(X)$ is an $\RR$-vector space.

By construction, and Proposition~\ref{chap4-prop-cpsh-vector-space}, it is a $\mathbb{Q}$-vector space and contains all constants. Let $u=\varphi_1-\varphi_2$ and $v = \psi_1-\psi_2$ be two d-cpsh functions with $\varphi_1, \varphi_2\in \cpsh(a\xi)$, and $\psi_1, \psi_2\in \cpsh(b\xi)$. Then 
\begin{align*}
    \max\{u,v\} &= \frac 12\left(u+v+|u-v|\right)
    =\frac 12\left(\varphi_1+\psi_1-\varphi_2-\psi_2+|\varphi_1+\psi_2-\varphi_2-\psi_1|\right)\\
    &= \frac 12 (\varphi_1+\psi_1-\varphi_2-\psi_2)\\
    & \ \ \  +\max\{\varphi_1+\psi_2, \psi_1+\varphi_2\}-2(\varphi_1+\psi_2+\psi_1+\varphi_2)\\
    &\in \dcpsh_\xi(X).
\end{align*}
It remains to show that it separates points. Let $v_1, v_2$ be two different points in $X_0$. We may assume $v_1(f_\alpha)<v_2(f_\alpha)$ for some $f_\alpha\in R_\alpha$. Let $t$ be such that $v_2(f_\alpha)-v_1(f_\alpha) = t\langle \xi, \alpha\rangle$. Then there is some $f_\beta\in R_\beta, \beta\neq \alpha$ such that $v_2(f_\beta)-v_1(f_\beta)\neq t\langle \xi, \beta\rangle$, since otherwise $v_2=(t\xi)*v_1$, but this implies $v_2\notin X_0$, contradiction. Let $\psi=\max\{\frac{\log|f_\gamma|+\lambda_\gamma}{\la \xi, \gamma\ra}\}$ be a function in $\fs(\xi)$ with $f_\beta = f_\gamma, \lambda_\beta=0$ for some $\gamma$, and choose $\lambda_\gamma, \gamma\neq \beta$ small enough  such that $\psi(v_i)=\frac{\log|f_\beta|(v_i)}{\langle\xi, \beta\rangle}$ for $i=1, 2$.
Take $u=\max\{\frac{\log|f_\alpha|}{\langle\xi,\alpha\rangle}, \psi-m\}-\psi$. Then $u(v_1)\neq u(v_2)$ for $m\gg 0$. This completes the proof.
\end{proof}
 
\begin{thm}\label{density}
For any $\xi$, the space $\dcpsh(X, \xi)$ dense in $C^0(X^{\an}\setminus\{o\},\xi).$
\end{thm}

\begin{proof}
 Let $\varphi$ be a continuous $\xi$-equivariant function on $X^{\an}\setminus\{o\}$. Then $\varphi-\varphi_\xi$ is a $\xi$-invariant function, and thus can be viewed as a function in $C^0(X_0).$ By previous theorem, there is a sequence of $u_i = \varphi_i-\psi_i\in \dcpsh_\xi(X)$ with $\varphi_i, \psi_i\in \cpsh(X, a_i\xi)$ such that $u_i$ converges to $\varphi-\varphi_\xi$ locally uniformly. Thus the sequence $(\varphi_i+\varphi_\xi-\psi_i)$ converges to $\varphi$ locally uniformly.
\end{proof}

Difference of continous psh functions are differences of psh-approachable functions in the sense of Chambert--Loir and Ducros. This allows us to generalize the mixed Monge--Amp\`ere measures.
\begin{defn-prop}\label{chap4-def-prop-MA-meas}
    Let $\varphi_1,  \varphi_2, \cdots, \varphi_{n-1}$ be functions in $\dcpsh(X, \xi)$ or $\dcpsh_\xi(X)$. Then the assignment
    \[(\varphi_1, \cdots, \varphi_{n-1})\mapsto d'd''\varphi_1\wedge \cdots \wedge d'd''\varphi_{n-1}\wedge d'd''\varphi_\xi^+\]
    is symmetric, multilinear with respect to convex combinations, and defines a Radon measure on $X^{\an}$ supported on $X_0$.
\end{defn-prop}

\begin{proof}
    First, when all the $\varphi_i$ are continuous psh functions, they are psh-approachable in the sense of Chambert--Loir and Ducros. The conclusion follows from~\cite[Corollaire 5.6.6]{CLD12}, since the set of continuous psh functions with a fixed polarization is closed under convex combination. In general, the $\varphi_i$ are differences of psh-approachable functions, and are again closed under convex combinations. Thus the result again follows from~\cite{CLD12}. That the support is in $X_0$ follows from Proposition~\cite[Proposition 5.5]{wu}.
\end{proof}

Similar to the global theory, the key properties for functions in $\dcpsh(X, \xi)$ and $\dcpsh_\xi(X)$ are the integration by parts formula and a local version of the Hodge index theorem. 

\begin{prop}[Integration by parts]\label{prop-int-by-parts}
Let $u, v\in \dcpsh_\xi(X)$, and $\varphi_1,\cdots, \varphi_{n-2}\in \dcpsh(X, \xi)$. Then
\[\int_{X^{\an}} u d'd''v\wedge d'd''\varphi_1\wedge \cdots \wedge d'd''\varphi_{n-2}\wedge d'd''\varphi_\xi^+ =\int_{X^{\an}} v d'd''u\wedge d'd''\varphi_1\wedge \cdots \wedge d'd''\varphi_{n-2}\wedge d'd''\varphi_\xi^+.  \]
\end{prop}
\begin{proof}
As in the proof of~\cite[Proposition 5.5]{wu}, all functions above can be approximated by smooth psh functions in the sense of~\cite{CLD12}. If $u,v, \varphi_1, \cdots, \varphi_{n-2}$ are all smooth, then both sides are equal to 
\[\int_{\{\varphi_\xi\geq 0\}} \varphi_\xi d'd''u\wedge d'd''v\wedge d'd''\varphi_1\wedge \cdots \wedge d'd''\varphi_{n-2}.\]
In general, the equality follows from a similar argument as in~\cite{wu}.
\end{proof}
\begin{prop}[Hodge Index Theorem]
Let $u\in \dcpsh(X)$, and $\varphi_1, \cdots, \varphi_{n-2}\in \cpsh(X, \xi)$. Then
\[\int_{X^{\an}} u d'd''u\wedge d'd''\varphi_1\wedge \cdots \wedge d'd''\varphi_{n-2}\wedge d'd''\varphi_\xi^+\leq 0.\]
\end{prop}
\begin{proof}
After possibly scaling by some positive constant, assume $u=\varphi-\psi$ for some $\varphi, \psi\in \cpsh(\xi)$. Pick $\varphi_i, \psi_i\in\fs(\xi_i)$ approximating $\varphi, \psi$ with $\xi_i$ rational. It then follows from~\cite[Proposition 3.5]{BJ18v1} and~\cite[Proposition 5.13]{wu} that the above inequality is true for all $\varphi_i, \psi_i$. 
We are done by letting $i\to \infty$.
\end{proof}

This allows us to introduce seminorms on $\dcpsh_\xi(X)$:
\begin{defn}
    For $u\in \dcpsh_\xi(X)$, and $\varphi_1, \cdots, \varphi_{n-2}\in \cpsh(\xi)$, define
    \[\|u\|_{(\varphi_1, \cdots, \varphi_{n-2})}\coloneqq \left(-\int_{X^{\an}} u d'd''u \wedge d'd''\varphi_1\wedge \cdots \wedge d'd''\varphi_{n-2}\wedge d'd''\varphi_\xi^+\right)^{\frac 12}.\]
    We will write $\|u\|_{(\varphi^j, \varphi'^{n-2-j})}$ if there are $j$ copies of $\varphi$, and $n-2-j$ copies of $\varphi'$.
\end{defn}
\begin{cor}[Cauchy-Schwarz Inequality]\label{Cor-CSInequality}
    Let $u,v\in \dcpsh_\xi(X)$, and $\varphi_1, \cdots, \varphi_{n-2}\in \cpsh(\xi)$. Then 
    \begin{align*}
        &\left|\int_{X^{\an}} u d'd''v \wedge d'd''\varphi_1\wedge \cdots \wedge d'd''\varphi_{n-2}\wedge d'd''\varphi_\xi^+\right|\leq \|u\|_{(\varphi_1, \cdots, \varphi_{n-2})}\|v\|_{(\varphi_1, \cdots, \varphi_{n-2})}.
    \end{align*}
\end{cor}
\begin{proof}
    By the previous proposition, we have that the symmetric bilinear form on $\dcpsh_\xi(X)$ given by 
    \[(u, v)\mapsto -\int_{X^{\an}} u d'd''v \wedge d'd''\varphi_1\wedge \cdots \wedge d'd''\varphi_{n-2}\wedge d'd''\varphi_\xi^+\]
    is semipositive definite. Hence the inequality follows.
\end{proof}

\section{Energy functionals on continuous plurisubharmonic functions}\label{section: EnergyOnPSH}
\subsection{Monge--Amp\`ere Energy}
We are now in the place to extend the Monge--Amp\`ere operator defined in~\cite{wu} to the class of continuous psh functions. 
\begin{defn}
The Monge--Amp\`ere energy of a d-cpsh function $\varphi\in\dcpsh(X, \xi)$ with respect to $\psi\in \dcpsh(X, \xi)$ is defined as 
\begin{align}\label{eqn:ma}
\E_\xi(\varphi, \psi) \coloneqq \frac{1}{n\vol(\xi)} \sum_{j=0}^{n-1}\int_{X^{\an}} (\varphi- \psi) (d'd''\varphi)^j\wedge (d'd''\psi)^{n-1-j}\wedge d'd''\varphi_{\xi}^+. 
\end{align}
We will omit $\xi$ if it is clear from context. We will simply write $\E(\varphi)$ if $\psi = \varphi_\xi$.
\end{defn}


We list below some properties of the Monge--Amp\`ere energy.
\begin{prop}\label{chap4-prop-ma-properties}
Let $\varphi,\psi, \varphi_1, \cdots, \varphi_{n-1}\in \cpsh(\xi)$, $c\in\mathbb{R}$, $u\in \dcpsh_\xi(X)$
\begin{enumerate}[(1)]
    \item The Monge--Amp\`ere measure
    $$\ma(\varphi_1, \cdots, \varphi_{n-1}) :=\frac{1}{\vol(\xi)} d'd''\varphi_1\wedge \cdots \wedge d'd''\varphi_{n-1}\wedge d'd''\varphi_{\xi}^+$$
    is a Radon probability measure supported on $X_0$.
    \item The terms in (\ref{eqn:ma}) are non-increasing in $j$.
    \item The first variation of $\E(\varphi)$ is
    $$ \qquad \odv{\E(\varphi+ tu)}{t}_{t=0}^{} = \frac{1}{\vol(\xi)}\int_{X^{\an}} u (d'd''\varphi)^{n-1}\wedge d'd''\varphi_{\xi}^+.$$
    \item $\E(\varphi)-\E(\psi) = \E(\varphi, \psi)$.
    \item If $\varphi\leq \psi$, then $\E(\varphi)\leq \E(\psi)$.
    \item $\E$ is concave on $\dcpsh(X, \xi)$.
    \item $\E(\varphi+c) = E(\varphi)+c$.
    \item $\E_{\frac 1a \xi}(a\varphi) = a \E_\xi(\varphi)$ for all positive real number $a$.
\end{enumerate}
\end{prop}

\begin{proof} By Definition-Proposition~\ref{chap4-def-prop-MA-meas}, we have proved (1), and (7) follows immediately. 
To prove (2), we apply Proposition~\ref{prop-int-by-parts}:
\begin{align*}
    &\int(\varphi-\psi) (d'd''\varphi)^j\wedge (d'd''\psi)^{n-1-j}\wedge d'd''\varphi_\xi^+ - \int(\varphi-\psi) (d'd''\varphi)^{j+1}\wedge (d'd''\psi)^{n-j-2}\wedge d'd''\varphi_\xi^+\\
    &=-\int(\varphi-\psi) d'd''(\varphi-\psi)\wedge (d'd''\varphi)^j\wedge (d'd''\psi)^{n-j-2}\wedge d'd''\varphi_\xi^+\geq 0.
\end{align*}
The third property is a direct calculation.
\begin{align*}
    &n\vol(\xi)(\E(\varphi+ tu)-\E(\varphi))\\
    &= \sum_{j=0}^{n-1} t(\int_{X^{\an}} u (d'd''\varphi)^j\wedge (d'd''\varphi_\xi)^{n-j-1}\wedge d'd''\varphi_\xi^+ \\
    &\indent + j\int_{X^{\an}} (\varphi-\varphi_\xi)(d'd''u)\wedge (d'd''\varphi)^{j-1}\wedge (d'd''\varphi_\xi)^{n-j-1}\wedge d'd''\varphi_\xi^+ +O(t^2) )\\
    &= t \sum_{j=0}^{n-1}(\int_{X^{\an}} u (d'd''\varphi)^j\wedge (d'd''\varphi_\xi)^{n-j-1}\wedge d'd''\varphi_\xi^+ \\
    &\indent + j\int_{X^{\an}} u(d'd''(\varphi-\varphi_\xi))\wedge (d'd''\varphi)^{j-1}\wedge (d'd''\varphi_\xi)^{n-j-1}\wedge d'd''\varphi_\xi^+ +O(t^2) )\\
    &= tn\int_{X^{\an}} u (d'd''\varphi)^{n-1}\wedge  d'd''\varphi_{\xi}^+ + O(t^2),
\end{align*}
where the second to last equality follows from integration by parts. 

Similar calculation shows
\[\qquad \odv[order=2]{\E(\varphi+ tu)}{t} = \frac{1}{\vol(\xi)}\int_{X^{\an}} u (d'd''(\varphi+tu))^{n-1}\wedge d'd''\varphi_\xi^+.\]
Now let $f(t) = \E(\varphi+t(\psi-\varphi))$, and $g(t):= \E(\varphi+t(\psi-\varphi),\psi)$ for $t\in [0,1]$. Then $f(0) = \E(\varphi)$, $f(1) = \E(\psi)$ and $g(0) = \E(\varphi,\psi)$, $g(1) = 0$. Further, the above calculation shows that 
$f' \equiv g'$ on $t\in [0,1]$.
Hence
\[\E(\varphi)-\E(\psi) = f(0)-f(1) = g(0)-g(1) = \E(\varphi, \psi).\]
This proves (4), and (5) follows from (4).
To get (6), we use the Hodge index theorem:
\[\qquad \odv[order=2]{\E(\varphi+tu)}{t} = \frac{n-1}{\vol(\xi)}\int_{X^{\an}} u d'd''u\wedge (d'd''(\varphi+tu))^{n-2}\wedge d'd''\varphi_\xi^+\leq 0.\]
Finally (8) follows from direct calculation and the fact that $\vol(\xi)$ is homogeneous of degree $-n$.
\end{proof}
\subsection{The I and J functionals}
We next introduce two related functionals defined using Monge--Amp\`ere measures. 
\begin{defn}
    For $\varphi, \psi\in \cpsh(\xi)$, we define 
    \[\I(\varphi, \psi)\coloneqq \int_{X^{\an}} (\varphi-\psi)(\ma(\psi)-\ma(\varphi)),\]
    and 
    \[\J_{\psi}(\varphi)\coloneqq \int_{X^{\an}} (\varphi-\psi) \ma(\psi) - \E(\varphi, \psi).\]
    We will simply write $\J(\varphi)$ if $\psi = \varphi_\xi$.
\end{defn}

\begin{prop}[{\cite[Lemma 3.26]{BJ18}}]\label{prop-IJ-formula}
    For $\varphi, \psi\in \cpsh(\xi)$, we have
    \begin{enumerate}
        \item $\I(\varphi, \psi)=\vol(\xi)^{-1} \sum_{j=0}^{n-2} \|\varphi-\psi\|^2_{(\varphi^j, \psi^{n-2-j})};$
        \item $\J_\psi(\varphi) = \vol(\xi)^{-1} \sum_{j=0}^{n-2}\frac{j+1}{n} \|\varphi-\psi\|^2_{(\varphi^j, \psi^{n-2-j})};$
        \item $\frac{1}{n} \I(\varphi, \psi)\leq J_\psi(\varphi)\leq \frac{n-1}{n} \I(\varphi,\psi).$
        \item $\I, \J$ are convex functionals on $\cpsh(\xi)$.
    \end{enumerate}
\end{prop}
\begin{proof}
    This is a direct computation as in Proposition~\ref{chap4-prop-ma-properties} (2):
    \begin{align*}
        &\vol(\xi)\I(\varphi, \psi) = \vol(\xi)\int (\varphi-\psi)(\ma(\psi)-\ma(\varphi))\\
        &= \sum_{j=1}^{n-2} \left(\int (\varphi-\psi) (d'd''\psi)^j\wedge (d'd''\varphi)^{n-1-j}\wedge d'd''\varphi_\xi^+ - \int (\varphi-\psi) (d'd''\psi)^{j-1}\wedge (d'd''\varphi)^{n-j}\wedge d'd''\varphi_\xi^+\right)\\
        &= \sum_{j=0}^{n-2} \int (\varphi-\psi) d'd''(\psi-\varphi)\wedge (d'd''\psi)^j\wedge (d'd''\varphi)^{n-2-j}\wedge d'd''\varphi_\xi^+
        = \sum_{j=0}^{n-2} \|\varphi-\psi\|_{(\varphi^j, \psi^{n-2-j})}.
    \end{align*}
    This proves (1). The proof of (2) is similar, and (3) follows from (1) and (2). Finally, (4) is a consequence of (3) and the fact that $E$ is concave.
\end{proof}

\begin{thm}[{\cite[Theorem 3.31]{BJ18}, \cite[Theorem 1.8]{BBEGZ}}, Quasi-triangle inequality]\label{thm-quasi-triangle-inequality}
    For $\varphi_1, \varphi_2, \varphi_3\in \cpsh(\xi)$,
    we have
    \[\I(\varphi_1, \varphi_2)\lesssim \I(\varphi_1, \varphi_3)+\I(\varphi_3,\varphi_1).\]
\end{thm}

\begin{lem}[{\cite[Lemma 3.33]{BJ18}}]\label{lem-norm-holder}
    For $\varphi, \varphi', \psi\in \cpsh(\xi)$, we have
    \[\|\varphi-\varphi'\|^2_{(\psi^{n-2})}\lesssim \vol(\xi)\I(\varphi, \varphi')^{\beta_n}\max\{\J_\varphi(\psi), \J_{\varphi'}(\psi)\}^{1-\beta_n},\]
    for $\beta_n = \frac{1}{2^{n-2}}$.
\end{lem}

\begin{proof}
    Put $\phi = \frac 12(\varphi+\varphi')$, $u = \varphi-\varphi'$, $A = \vol(\xi)\I(\varphi, \varphi')$ and $B = \vol(\xi) \max\{\J_\varphi(\psi), \J_{\varphi'}(\psi)\}$. For $0\leq p\leq n-2$, set
    \[b_p\coloneqq \|u\|_{(\psi^p, \phi^{n-2-p})}^2.\]
    It is clear that $b_0\leq A$. Writing $u=(\varphi-\psi)+(\psi-\varphi')$ yields
    \[b_{n-2}\leq \left(\|\varphi-\psi\|_{(\psi^{n-2})} +\|\varphi'-\psi\|_{(\psi^{n-2})}\right)^2\lesssim B,\]
    where the last inequality follows from Proposition~\ref{prop-IJ-formula}. If $A\geq B$, then $b_{n-2}\lesssim B\leq A^{\beta_n}B^{1-\beta_n}$, which is what we want. From now on, assume $A\leq B$, and we proceed by induction on $p$ to show that for each $0\leq p\leq n-2$, one has
    \[b_p\lesssim A^{\frac{1}{2^p}} B^{1-\frac{1}{2^p}}.\]
    Note that the base case $p=0$ is true, and it suffices to do the inductive step. For $0\leq p \leq n-3$, we have
    \begin{align*}
        b_{p+1}-b_p &= -\int u d'd''u (d'd''\psi)^p\wedge (d'd''\phi)^{n-3-p}\wedge d'd''(\psi-\phi)\wedge d'd''\varphi_\xi^+\\
        &= -\int u d'd''(\psi-\phi)\wedge d'd''\varphi\wedge (d'd''\psi)^p\wedge (d'd''\phi)^{n-3-p}\wedge d'd''\varphi_\xi^+\\
        &+ \int u d'd''(\psi-\phi)\wedge d'd''\varphi'\wedge (d'd''\psi)^p\wedge (d'd''\phi)^{n-3-p}\wedge d'd''\varphi_\xi^+\\
        &=: I+II.
    \end{align*}
    By Cauchy--Schwarz, and the fact that $d'd''\varphi\leq 2d'd''\phi$ as currents, we get
    \[|I|\lesssim \|u\|_{(\varphi, \psi^p, \phi^{n-3-p})}\|\psi-\phi\|_{(\varphi, \psi^p, \phi^{n-3-p})} \lesssim \sqrt{b_p} \sqrt{I(\psi, \phi)}\lesssim \sqrt{Bb_p}.\]
    For the same reason, we have the same bound for $|II|$ as well, and so 
    \[b_{p+1}-b_p\lesssim \sqrt{Bb_p}.\]
    Now by induction hypothesis, we have
    \begin{align*}
        b_{p+1}\lesssim b_p+\sqrt{Bb_p}&\lesssim A^{\frac{1}{2^p}} B^{1-\frac{1}{2^p}} + A^{\frac{1}{2^{p+1}}} B^{1-\frac{1}{2^{p+1}}} \\
        &=A^{\frac{1}{2^{p+1}}} B^{1-\frac{1}{2^{p+1}}}\left(\left(\frac{A}{B}\right)^{\frac{1}{2^{p+1}}}+1 \right)\lesssim A^{\frac{1}{2^{p+1}}} B^{1-\frac{1}{2^{p+1}}},
    \end{align*}
    where we have used $A\leq B$ in the last inequality. Setting $p=n-2$, we are done.
\end{proof}

\begin{proof}[Proof of Theorem~\ref{thm-quasi-triangle-inequality}]
    Let $\psi = \frac{\varphi_1+\varphi_2}{2}$. Direct calculation shows
    \begin{align*}
        \vol(\xi)\I(\varphi_1, \varphi_2)&\lesssim \|\varphi_1-\varphi_2\|_{(\psi^{n-2})}\lesssim \max\{\|\varphi_1-\varphi_3\|_{(\psi^{n-2})},\|\varphi_3-\varphi_2\|_{(\psi^{n-2})}\}.
    \end{align*}
    By Lemma~\ref{lem-norm-holder} and convexity, we have for $i=1,2$,
    \begin{align*}
    \|\varphi_i-\varphi_3\|_{(\psi^{n-2})} &\lesssim \vol(\xi)\I(\varphi_i, \varphi_3)^{\beta_n}\max\{\J_{\varphi_i}(\psi), \J_{\varphi_3}(\psi)\}^{1-\beta_n}\\
    &\lesssim \vol(\xi)\I(\varphi_i, \varphi_3)^{\beta_n}\max\{\I(\varphi_1, \varphi_2), \I(\varphi_3, \varphi_1), \I(\varphi_3, \varphi_2)\}^{1-\beta_n}.
    \end{align*}
    Putting these together, we have
    \[\I(\varphi_1, \varphi_2)\lesssim \max\{I(\varphi_1, \varphi_2), \I(\varphi_3, \varphi_1), \I(\varphi_3, \varphi_2)\}^{1-\beta_n}\max\{I(\varphi_1, \varphi_3), \I(\varphi_2, \varphi_3)\}^{\beta_n}.\]
    Hence the result follows.
    
\end{proof}

\begin{cor}[{\cite[Corollary 3.34]{BJ18}}]\label{cor-norm-holder}
    Let $\varphi, \varphi', \psi_1, \cdots, \psi_{n-2}\in \cpsh(\xi)$. Then
    \[\|\varphi-\varphi'\|^2_{(\psi_1, \cdots, \psi_{n-2})}\lesssim \vol(\xi) \I(\varphi, \varphi')^{\beta_n}\max\{\J(\varphi), \J(\varphi'), \J(\psi_1), \cdots, \J(\psi_{n-2})\}^{1-\beta_n},\]
    for $\beta_n = \frac{1}{2^{n-2}}$.
\end{cor}
\begin{proof}
    Put $\psi\coloneqq \frac{1}{n-2}\sum_{i=1}^{n-2}\psi_i\in \cpsh(\xi)$. Then the previous lemma gives
    \[\|\varphi-\varphi'\|_{(\psi_1, \cdots, \psi_{n-2})}^2\lesssim \|\varphi-\varphi'\|^2_{(\psi^{n-2})}\lesssim \vol(\xi)\I(\varphi, \varphi')^{\beta_n}\max\{\J_\varphi(\psi), \J_{\varphi'}(\psi)\}^{1-\beta_n}.  \]
    Now we conclude by the quasi-triangle inequality and Proposition~\ref{prop-IJ-formula} (3).
\end{proof}

\begin{prop}[{\cite[Lemma 7.30]{BJ18}}]\label{prop-holder-cpsh}
    For all $\varphi, \varphi', \psi, \psi'\in \cpsh(\xi)$, we have
    \[\left|\int (\varphi-\varphi')(\ma(\psi)-\ma(\psi'))\right|\lesssim \I(\varphi, \varphi')^{\alpha_n}\I(\psi, \psi')^{\frac 12}\max\{\J(\varphi), \J(\varphi'), \J(\psi), \J(\psi')\}^{\frac 12-\alpha_n},\]
    for $\alpha_n = \frac{1}{2^{n-1}}$.
\end{prop}

\begin{proof}
    Direct computation as in the proof of Proposition~\ref{chap4-prop-ma-properties} (2) and Proposition~\ref{prop-IJ-formula} shows
    \begin{align*}
        &\vol(\xi)\int (\varphi-\varphi')(\ma(\psi)-\ma(\psi'))\\
        &= \sum_{j=0}^{n-2} \int (\varphi-\varphi') d'd''(\psi-\psi')\wedge (d'd''\psi)^j\wedge (d'd''\psi')^{n-2-j}\wedge d'd''\varphi_\xi^+.
    \end{align*}
    By the Cauchy--Schwarz Inequality (Corollary~\ref{Cor-CSInequality}), we infer
    \begin{align*}
        \left|\int (\varphi-\varphi')(\ma(\psi)-\ma(\psi'))\right|
        \leq \vol(\xi)^{-1}\max_j \left\{\|\varphi-\varphi'\|_{(\psi^j, \psi'^{n-2-j})}\|\psi-\psi'\|_{(\psi^j, \psi'^{n-2-j})}\right\}.
    \end{align*}
    Now using Proposition~\ref{prop-IJ-formula} and Corollary~\ref{cor-norm-holder}, we have
    \begin{align*}
        &\max_j \left\{\|\varphi-\varphi'\|_{(\psi^j, \psi'^{n-2-j})}\|\psi-\psi'\|_{(\psi^j, \psi'^{n-2-j})}\right\}
        \lesssim \vol(\xi)^{\frac 12}\max_j \left\{\|\varphi-\varphi'\|_{(\psi^j, \psi'^{n-2-j})}\right\}\J(\psi, \psi')^{\frac 12}\\
        &\lesssim \vol(\xi) \I(\varphi, \varphi')^{\alpha_n} \J(\psi, \psi')^{\frac 12} \max\{\J(\varphi), \J(\varphi'), \J(\psi), \J(\psi')\}^{\frac 12-\alpha_n}.
    \end{align*}

\end{proof}

\section{General plurisubharmonic functions}\label{section: generalPSH}
In this section, we extend our class of psh functions even further to allow only upper-semicontinuous (usc) functions. We also extend the Monge-Amp\`ere operator to this class. The main goal is to show that this extended Monge-Amp\`ere energy still has nice continuity properties. 

\begin{defn}
We call $\varphi: X^{\an}\to \mathbb{R}\cup \{-\infty\}$ a \emph{$\xi$-plurisubharmonic ($\xi$-psh) function} on $(X^{\an}; \xi)$ if $\varphi$ is a pointwise limit of a decreasing net $(\varphi_i)$ in $\cpsh(\xi)$, and $\varphi \not\equiv -\infty$. We denote by $\psh(\xi)$ the set of $\xi$-psh functions on $X^{\an}$.
\end{defn}

\begin{prop}\label{prop-psh-fs}
Fix $\xi$ on $X$.
\begin{enumerate}
    \item $\psh(a\xi) = \frac 1a \psh(\xi)$ for all $t\in \mathbb{R_+}$.
    \item $\psh(\xi)$ is convex and $C^0(X^{\an}\setminus\{0\},\xi)\bigcap \psh(\xi) =\cpsh(\xi)$.
    \item For any decreasing net $(\varphi_i)$ in $\psh(\xi)$ with pointwise limit $\varphi$, either $\varphi\in\psh(\xi)$ or $\varphi \equiv -\infty$.
    \item Every function in $\psh(\xi)$ is the pointwise limit of a decreasing net in $\cH(\xi)$.
\end{enumerate}
\end{prop}
\begin{proof}
(1) and first assertion of (2) follow from analogous result on continous psh functions. It is clear that $\cpsh(\xi)\subseteq C^0(X^{\an}\setminus \{o\}, \xi)\cap \psh(\xi)$. Let $\varphi\in \psh(\xi)$ be a function that is also continuous on $X^{\an}\setminus \{o\}$. Then by definition, there is a decreasing sequence $\varphi_i\in \cpsh(\xi)$ converging pointwise to $\varphi$. Then on each compact subset $K\subset X^{\an}\setminus\{o\}$, the convergence is uniform by Dini's lemma. Hence $C^0(X^{\an}\setminus \{o\}, \xi)\cap \psh(\xi)=\cpsh(\xi)$.
(3) and (4) are consequences of~\cite[Lemma 4.6]{BJ18} with $K = X_0, \widetilde{\mathcal{F}} = \cpsh(\xi), \mathcal{F} = \cH(\xi)$.
\end{proof}

\begin{cor}\label{chap4-cor-qm-linear-growth}
    For any $\varphi\in \psh(\xi)$ we have $\varphi>-\infty$ on the set of quasi-monomial valuations in $X_0$.
\end{cor}
\begin{proof}
    First if $\varphi = \frac{\log|f_\alpha|+\lambda_\alpha}{\la \xi, \alpha\ra}$, then by Proposition~\ref{chap3-prop-T-formula} we have
    \[\varphi(v)\geq \varphi(v_{\triv})-T(v;\xi).\]
    We claim that the same inequality holds for $\varphi\in \psh(\xi)$.
    If $\varphi\in \fs(\xi)$, then the above inequality still holds, since taking finite maxima does not affect it. In general, it follows from Definition-Proposition~\ref{prop-psh-fs}. Thus it remains to show that $T(v;\xi)<\infty$. To this end, note that finiteness does not depend on $\xi$, we can choose a compactification $Y\supset X$ where $Y$ is a normal projective variety with grading induced by some rational $\xi$ (e.g. one can take $Y$ to be the corresponding orbifold cone for some rational $\xi$). Then by~\cite[Proposition 2.12]{vallinear}, $v$ has linear growth, that is, $T(v;\xi)<\infty$.
\end{proof}

We in fact have an alternative characterization of $T(v)$ using general psh functions.
\begin{cor}\label{chap4-cor-T-psh-duality}
    For any $v\in X_0$, we have 
    \[T(v, \xi) = \sup_{\varphi\in \psh(\xi)} \left(\sup_{X_0}\varphi - \varphi(v).\right)\]
\end{cor}
\begin{proof}
    Denote by $T$ the right hand side of the equality in the corollary. 
    Let $\varphi = \frac{\log|f_\alpha|}{\la\xi,\alpha\ra}\in \psh(\xi)$. Then
    $\varphi(v)\geq \varphi(v_{\triv})-T(v)$
    implies $T(v, \xi)\leq T$ by Definition-Proposition~\ref{chap3-prop-T-formula}.
    On the other hand, after adding constants and taking max, the same inequality holds for $\varphi\in \fs(\xi)$, and hence also for $\varphi\in \psh(\xi)$ by Proposition~\ref{prop-psh-fs}. This proves $T(v; \xi)\geq T$, and we are done. 
\end{proof}

We equip $\psh(\xi)$ with the topology of pointwise convergence on quasi-monomial valuations. The following is an analog of~\cite[Lemma 4.26]{BJ18}.
\begin{lem}\label{lem-fs-finite-qm}
Let $\psi\in \cH(\xi)$.
Then there is a finite set $S\subset X^{\an}$ of quasi-monomial valuations such that for any $\varphi\in \cH(\xi)$,
\[\sup_{X^{\an}\setminus\{o\}}  (\varphi-\psi) = \sup_S (\varphi-\psi).\]
\end{lem}

\begin{proof}
Assume $\psi = \max\limits_{i=1, \cdots, N} \{\frac{\log|f_i|+\lambda_i}{\langle \xi, \alpha_i\rangle}\}$, and by possibly adding more functions $f_j$ with sufficiently small $\lambda_j$, we may assume $(f_1, \cdots, f_N)$ generate the maximal ideal $\fm$ at $o$.
Let $Y$ be a  $\mathbb{T}$-equivariant log resolution of the ideal $\fm = (f_1, \cdots, f_N)$, and denote the map by $\pi: Y\to X$. Then $\pi^{-1}(0)$ is covered by finitely many open subsets $\{U_k\}_{k=1}^s$ such that on each $U_k$ one can choose local coordinates $z_1, \cdots, z_r$ with $r\leq n$ such that $\pi^*(f_i) = u_i\prod_{j=1}^r z_j^{a_{ij}}$, where the $u_i$ are units. Let $W: = \bigcup_{k=1}^s \mathrm{red}^{-1}(U_k)\subset Y^{\an}$. Now by $\xi$-invariance of $\varphi-\psi$, $\sup_{X^{\an}} (\varphi -\psi)$ is achieved on the subset of points centered at $o\in X$. Hence
\[\sup_{X^{\an}} (\varphi -\psi) = \sup_{W} \pi^*(\varphi - \psi).\] 
It remains to show there is some finite subset $S_k\subset \mathrm{red}^{-1}(U_k)$ only depending on $\psi$ such that 
\[\sup_{\mathrm{red}^{-1}(U_k)} \pi^*(\varphi - \psi) = \sup_{S_k} \pi^*(\varphi - \psi).\]
To this end, to simplify the notation, we assume that we are on some $U = U_k$, and choose $z_1, \cdots, z_r$ as before. Let $x_j = \log|z_j|\in [-\infty, 0), 1\leq j\le r$, $\boldsymbol{x} = (x_1, \cdots, x_r)$. Write $\ell_i(\boldsymbol{x}) = \frac{1}{\langle \xi, \alpha_i\rangle}\sum_{j=1}^r a_{ij}x_j$. Then 
\[\pi^*\psi = \max_{1\leq i\leq N} \{\ell_i(\boldsymbol{x}) +\frac{\lambda_i}{\langle\xi, \alpha_i\rangle}\}\]
Without loss of generality, we can assume $\varphi$ takes the form $\varphi = \frac{\log|g|}{\langle\xi, \beta\rangle}+\mu$, where $\pi^*g$ is a monomial in $f_1, \cdots, f_N$. Indeed, say $\pi^*g = \sum_{i=1}^M w_i f_1^{b_{i1}}\cdots f_N^{b_{iN}}$. Then 
\[\log|\pi^*g| \leq \max_{1\leq i\leq M} \log|f_1^{b_{i1}}\cdots f_N^{b_{iN}}|. \]
Thus we can cover $\mathrm{red}^{-1}(U_k)$ by finitely many samller opens on which $\log|\pi^*g| = \log|f_1^{b_1}\cdots f_N^{b_N}|$ for some monomial appearing in the above. Hence we may assume $\pi^*g = f_1^{b_1}\cdots f_N^{b_N}$. Then
\[\pi^*\varphi = \frac{\sum_{i=1}^N b_i\log|f_i|}{\langle \xi, \beta\rangle}+\mu = \sum_{i=1}^N\frac{b_i \langle\xi, \alpha_i\rangle \ell_i(\boldsymbol{x})}{\sum_{j=1}^N b_j\langle\xi, \alpha_j\rangle}+\mu.\]
Let $c_i = \frac{b_i \langle\xi, \alpha_i\rangle }{\sum_{j=1}^N b_j\langle\xi, \alpha_j\rangle}, 1\leq i\leq N.$ Note that $c_i\geq 0, \forall i$ and $\sum_{i=1}^N c_i = 1.$
Now on the subset $$V_1:= \{\boldsymbol{x} = (x_1, \cdots, x_r)\in \mathbb{R}^r: \pi^*\psi = \ell_1(\boldsymbol{x})+ \frac{\lambda_1}{\langle \xi, \alpha_1\rangle}, x_i\leq -1, \  \forall i=1, \cdots, r\},$$ we can write 
\[\pi^*\varphi - \pi^*\psi = \sum_{i=1}^N c_i\ell_i(\boldsymbol{x}) - \ell_1(\boldsymbol{x}) +\mu-\frac{\lambda_1}{\langle \xi, \alpha_1\rangle} = \sum_{i=2}^N c_i (\ell_i(\boldsymbol{x}) - \ell_1(\boldsymbol{x})) +\mu-\frac{\lambda_1}{\langle \xi, \alpha_1\rangle}.\]
We have now reduced the problem to showing there is some finite subset $S_{k1}\subset V_1$ depending only on $\psi$ such that
\[\sup_{V_1}(\pi^*\varphi - \pi^*\psi) = \sup_{S_{k1}} (\pi^*\varphi - \pi^*\psi),\]
for any $\pi^*\varphi$ of the above form, since any point in $V_1$ corresponds to a quasi-monomial valuation by~\cite{JM12}.
This is done by the following lemma.
\end{proof}

\begin{lem}
Let $\ell_1(\boldsymbol{x}), \cdots, \ell_N(\boldsymbol{x})$ be linear functions on $\mathbb{R}^r$ with non-negative coefficients, and $\boldsymbol{v}$-equivariant in the sense that 
\[\ell_i(\boldsymbol{x}-t\boldsymbol{v}) = \ell_i(\boldsymbol{x}) -t,  1\leq i\leq N,\]
for some nonzero $\boldsymbol{v}\in \mathbb{R}^r_{\geq 0}$. Let $\lambda_1, \cdots, \lambda_N$ be real numbers. Set
\[\Omega = \{\boldsymbol{x} = (x_1, \cdots, x_r)\in \mathbb{R}^r: \ell_1(\boldsymbol{x})+\lambda_1\geq \ell_i(\boldsymbol{x})+\lambda_i, \  2\leq i \leq N, \ x_j\leq -1,  \  1\leq j\leq r\}\]
Assume $\ell_1(\boldsymbol{x})+\lambda_1>-\infty$ on $\Omega$.
Then there is a finite subset $S\subset \Omega$ such that for any function $f$ of the form
\[f(\boldsymbol{x}) = \sum_{i=2}^{N}c_i(\ell_i(\boldsymbol{x})-\ell_1(\boldsymbol{x}))\]
where $c_2, \cdots, c_N$ are nonnegative real numbers with $\sum_{i=2}^N c_i\leq 1$, one has
\[\sup_{\Omega} f = \sup_S f.\]
\end{lem}

\begin{proof}
By assumption, $\sup_\Omega f\leq 0$.
If $\Omega$ is empty, then we are done. Assume that $\Omega$ is non-empty. Note that $\Omega$ is a (possibly unbounded) convex polyhedron, so any point $\boldsymbol{x}\in\mathrm{int} \  \Omega$ is a convex combination of two points in $\partial \Omega$. Thus $\sup_\Omega f = \sup_{\partial \Omega} f$. An induction on the dimension of $\Omega$ shows that 
$\sup_\Omega f $ is achieved on vertices or points in infinite one dimensional rays in $\partial \Omega.$ We set
$$S=\{\mathrm{vertices \ in \ }\partial \Omega\}\bigcup \bigcup_{\boldsymbol{r}: \mathrm{\ infinite \ ray}}S_{\boldsymbol{r}},$$ 
with $S_{\boldsymbol{r}}$ to be chosen.
Let $\boldsymbol{r} = \boldsymbol{x_0}+ t\boldsymbol{w} = (x_{01}+tw_1, \cdots, x_{0r}+tw_r)$ be an infinite ray. Assume $\sup f$ is achieved on this ray. The claim is that there is some finite set $S_{\boldsymbol{r}}$ on which the supremum is attained. We proceed by induction on $r$. If $r=1$, then $f$ is constant, and we can choose $S = \{x=-1\}$. For $r>1$, there are two possibilities. 
\begin{enumerate}
    \item If $w_i=0$ for some $i$, say $w_1 = 0$, then $\ell_i$ restricted to $\mathbb{R}^{r-1}\cong \{\boldsymbol{x}\in \mathbb{R}^r: x_1 = x_{01}$ is again a linear function that is $\boldsymbol{v'}$-equivariant for some vector $\boldsymbol{v'}\in \mathbb{R}^{r-1}$, and for all $1\leq i\leq N$. By the induction hypothesis, we can choose a finite set $S_{\boldsymbol{r}}$ on which any function $f$ of the above form attains its max.
    \item Now assume $w_i\neq 0, 1\leq i \leq r$. Then $\boldsymbol{w} = c\boldsymbol{v}$ for some nonzero constant $c$, and $f$ takes the same value along $\boldsymbol{r}$ due to $\boldsymbol{v}$-invariance. Thus we can take $S_{\boldsymbol{r}} = \{\boldsymbol{x_0}\}$;
\end{enumerate}

Therefore, we have shown that $S$ chosen in the above way is finite. 
\end{proof}
\begin{cor}\label{usc_lem}
For any $\psi\in \cpsh(\xi)$, $\varphi\mapsto \sup_{X^{\an}} (\varphi-\psi)$ is continuous on $\psh(\xi)$.
\end{cor}
\begin{proof}
First assume $\psi\in\fs(\xi)$. Then this is a direct consequence of Lemma~\ref{lem-fs-finite-qm}. In general, let $\psi_i$ be a sequence of FS functions converging locally uniformly to $\psi$. Let $\varphi_j$ be a net of psh functions converging pointwise on quasi-monomial valuations to $\varphi$. Fix $\varepsilon>0$. Then for $i\gg 0$, $\sup_{X^{\an}}(\psi_i-\psi) = \sup_{X_0}(\psi_i-\psi)<\varepsilon$. Fix a such $i$. Then there is a finite set $S_i$ of quasi-monomial valuations such that
\[\sup_{X^{\an}}(\varphi_j-\psi)\leq \sup_{X^{\an}}(\varphi_j-\psi_i)+\sup_{X^{\an}}(\psi_i-\psi)  < \sup_{S_i}(\varphi_j-\psi_i)+\varepsilon.\]
Thus for $j\gg 0$ such that $\sup_{S_i}(\varphi_j-\varphi)<\varepsilon$, one has
\[\sup_{X^{\an}}(\varphi_j-\psi)< \sup_{S_i}(\varphi_j-\varphi)+\sup_{S_i}(\varphi-\psi_i)+\varepsilon< 2\varepsilon+\sup_{X^{\an}}(\varphi - \psi)+\sup_{X^{\an}}(\psi - \psi_i)< 3\varepsilon + \sup_{X^{\an}}(\varphi-\psi).\]
By symmetry, we can conclude.
\end{proof}

\begin{thm}
Given $\varphi, \psi\in \psh(\xi)$, $\varphi\leq \psi$ on $X^{\an}$ if and only if $\varphi\leq \psi$ on $X^{\qm}$. In particular, the topology of $\psh(\xi)$ is Hausdorff.
\end{thm}
\begin{proof}
If $\psi\in\fs(\xi)$, then this is a consequence of Lemma~\ref{lem-fs-finite-qm}. Arguing as in the previous corollary, this is true for $\psi\in \cpsh(\xi)$ since $\cpsh(\xi)$ is the closure of $\fs(\xi)$ under local uniform convergence. In general, pick a decreasing net $(\psi_i)$ in $\cpsh(\xi)$ with $\psi_i\to \psi$. Then $\varphi\leq \psi\leq \psi_i$ on $X^{\qm}$ implies $\varphi\leq \psi_i$ on $X^{\an}$ and hence $\varphi\leq \psi$ on $X^{\an}$.
\end{proof}

Finally, we define the Monge--Amp\`ere energy for general psh functions.
\begin{defn}
Let $\varphi\in \psh(\xi)$. The Monge--Amp\`ere energy for $\varphi$ is defined as
\[\E(\varphi) = \E_\xi(\varphi)\coloneqq \inf\{\E(\psi):  \varphi\leq \psi, \psi\in\cpsh(\xi)\}.\]
\end{defn}
In view of Proposition~\ref{prop-psh-fs}, we can also take $\psi\in \cH(\xi)$ in the above definition. The next proposition summarizes the main properties of $\E$.
\begin{prop}\label{chap4-prop-ma-energy-psh}
The Monge-Amp\`ere energy $E: \psh(\xi)\to \mathbb{R}\cup \{-\infty\}$ satisfies the following properties:
\begin{enumerate}
    \item $\E$ is non-decreasing.
    \item $\E(\varphi+c) = E(\varphi)+c$.
    \item $\E$ is upper semicontinuous.
    \item $\E$ is continuous along decreasing nets.
    \item $\E$ is concave.
\end{enumerate}
\end{prop}
\begin{proof}
The first two assertions are obvious, and (4) is a consequence of upper semicontinuity and monotonicity. (5) follows from (4) and the fact that $E$ is concave on $\cpsh(\xi)$.
It remains to show upper semicontinuity. The proof goes along the same lines as in~\cite[Theorem 7.1]{BJ18}. Suppose that $\varphi\in \psh(\xi)$ and $t\in \mathbb{R}$ satisfies $E(\varphi)<t$. We can choose $\psi\in \cpsh(\xi)$ and $\varepsilon>0$ such that $\varphi\leq \psi$ and $E(\psi)<t-\varepsilon$. By Proposition~\ref{usc_lem}, we can find an open neighborhood $U$ of $\varphi$ in $\psh(\xi)$ such that $\varphi'\leq \psi+\varepsilon$ for all $\varphi'\in U$. Monotonicity then gives
\[E(\varphi')\leq E(\psi+\varepsilon) = t-\varepsilon+\varepsilon = t,\]
for all $\varphi'\in U$. This proves (3).
\end{proof}

\section{Functions of finite energy and mixed MA measures}\label{section: ExtendMA}
As suggested by the work of Berman--Boucksom--Jonsson~\cite{BBJ}, the non-Archimedean version of the finite energy class can play an important role in geometric applications. In the local case, we have a similar theory of functions of finite energy.
\begin{defn}
We say $\varphi\in \psh(\xi)$ is a function of \emph{finite energy} if $\E(\varphi)>-\infty$. The space of functions of finite energy is denoted by $\mathcal{E}^1(\xi).$
\end{defn}
The following is a direct consequence of Proposition~\ref{chap4-prop-ma-energy-psh}. 
\begin{prop}
The space $\mathcal{E}^1(\xi)$ satisfies the following:
\begin{enumerate}
    \item If $\varphi\in \mathcal{E}^1(\xi)$ and $c\in\mathbb{R}$, then $\varphi+c\in \mathcal{E}^1(\xi).$
    \item $\mathcal{E}^1(\xi)$ is convex.
    \item If $\varphi\leq \psi$ and $\varphi\in \mathcal{E}^1(\xi)$, then $\psi\in \mathcal{E}^1(\xi)$.
    \item $\cpsh(\xi)\subset \mathcal{E}^1(\xi).$
\end{enumerate}
\end{prop}

\subsection{Mixed Monge--Amp\`ere measures}
We are now ready to establish an analogue of~\cite[Theorem 6.9]{BJ18v1}, and the proof is the same. We include a proof for the sake of completeness. Here is the version of the theorem we will prove. This in particular is the main part of Theorem~\ref{main-thm-pluripotential}.
\begin{thm}\label{chap4-thm-mixed-ma}
Let $p$ be any integer such that $0\leq p\leq n-1$.
For any given $(n-1)$-tuple $(\varphi_1,\cdots,\varphi_{n-1})$ with $\varphi_1, \cdots, \varphi_p\in \mathcal{E}^1(\xi)$, and $\varphi_{p+1}, \cdots, \varphi_{n-1}\in \fs(\xi)$, 
there exists a unique Radon probability measure $\ma(\varphi_1, \cdots, \varphi_{n-1})$ on $X^{\an}\setminus\{0\}$ supported on $X_0$ such that the following are true:
\begin{enumerate}[(1)]
    \item If $\varphi_1,\cdots, \varphi_{p}\in \cpsh(\xi),$
    $$\ma(\varphi_1, \cdots, \varphi_{n-1}) =\vol(\xi)^{-1} d'd''\varphi_1\wedge\cdots \wedge d'd''\varphi_{n-1}\wedge d'd''\varphi_\xi^+.$$
    \item $\int_{X^{\an}} (\psi-\varphi_\xi) \ma(\varphi_1, \cdots, \varphi_{n-1}) >-\infty$ when $\psi, \varphi_1, \cdots, \varphi_{p}\in \mathcal{E}^1(\xi).$
    \item the pairing
    \[(\psi, \varphi_1, \cdots, \varphi_{p})\mapsto 
    \int_{X^{\an}} \psi \ma (\varphi_1,\cdots, \varphi_{n-1})\]
    is continuous along decreasing nets in $\mathcal{E}^1(\xi).$
\end{enumerate}
In particular, when $p=n-1$, the theorem defines a mixed Monge--Amp\`ere measure on $\cE^1(\xi)$.
\end{thm}

Before proving the theorem, we need the following auxilliary lemma. 
\begin{lem}\label{chap4-lem-finite-on-E1}
    Let $\varphi_0, \cdots, \varphi_{n-1}\in \cH(\xi)$. Then
    \[\int_{X^{\an}} \varphi_0 \ma(\varphi_1, \cdots, \varphi_{n-1})\geq \int_{X^{\an}}\varphi_0\ma(\varphi_\xi) -C_n \max \J(\varphi_i)\]
    for some positive dimensional constant $C_n$, where 
    $\J(\varphi) = \int_{X^{\an}} \varphi\ma(\varphi_\xi) =\varphi(v_{\triv})-\E(\varphi)$.
\end{lem}
\begin{proof}
    We may assume $\max \varphi_i =0$ on $X_0$ for all $i$. Write $\phi\coloneqq \frac 1n \sum_{i=0}^{n-1}\varphi_i$. Since $E$ is concave, we have
    \[\E(\varphi)\geq \frac 1n \sum_{i=0}^{n-1} \E(\varphi_i)\geq \min_i\E(\varphi_i)=\min_i\{\int_{X^{\an}} \varphi_i\ma(\varphi_\xi)-\J(\varphi_i)\}\geq-\max_i\J(\varphi_i).\]
    On the other hand, by Proposition~\ref{chap4-prop-ma-properties} (2), 
    \begin{align*}
        \E(\varphi)&=\frac{1}{n\vol(\xi)}\sum_{j=0}^n \int_{X^{\an}} \varphi (d'd''\varphi)^j\wedge (d'd''\varphi_\xi)^{n-1-j}\wedge d'd''\varphi^+\\
        &\leq \frac{1}{\vol(\xi)}\int_{X^{\an}} \varphi (d'd''\varphi)^{n-1}\wedge d'd''\varphi^+
        \leq \frac{1}{n\vol(\xi)}\int_{X^{\an}} \varphi_0(d'd''\varphi)^{n-1}\wedge d'd''\varphi^+\\
        &\leq \frac{n!}{n^{n+1}}\int_{X^{\an}} \varphi_0\ma(\varphi_1, \cdots, \varphi_{n-1}).
    \end{align*}
\end{proof}

\begin{proof}[Proof of Theorem~\ref{chap4-thm-mixed-ma}]
As in~\cite{BJ18v1}, we prove existence by induction on $p$.
When $p=0$, $\ma(\varphi_1, \cdots, \varphi_n)$ is defined as in (1) and it suffices to show that for any $\psi\in \psh(\xi)$, $\psi>-\infty$ on the set of quasi-monomial valuations in $X_0$. This was proved in Corollary~\ref{chap4-cor-qm-linear-growth}.

Now assume the theorem is true for $p-1$. Define $\ma (\varphi_1, \cdots, \varphi_{n-1})$ inductively as follows: 
\begin{align*}
\int_{X^{\an}} (\psi-\varphi_\xi) \ma (\varphi_1, \cdots, \varphi_{n-1})&: = \int_{X^{\an}} (\varphi_p-\varphi_\xi) \ma (\varphi_1, \cdots, \varphi_{p-1}, \psi, \varphi_{p+1}, \cdots, \varphi_{n-1} )\\
&+ \int_{X^{\an}} (\psi-\varphi_p) \ma (\varphi_1, \cdots, \varphi_{p-1}, \varphi_\xi, \varphi_{p+1}, \cdots, \varphi_{n-1} )
\end{align*}
for all $\psi\in \fs(\xi)$.
The right hand side is well-defined by induction.
Thus the measure $\ma(\varphi_1, \cdots,\varphi_{n-1})$ is a well-defined Radon probability measure by density of $\dcpsh(X)$ in $C^0(X_0)$. It is clear that when $\varphi_1, \cdots, \varphi_{p}\in\cpsh(\xi)$, the above equality holds as a result of integration by parts. Furthermore, the measure is continuous along decreasing nets of $(\varphi_1, \cdots, \varphi_p)$.

To prove (2), we may assume $\psi, \varphi_i\leq 0$. Set 
\[B\coloneqq \max\{-\E(\psi), -\E(\varphi_1), \cdots, -\E(\varphi_{n-1})\}<\infty.\]
Let $(\psi^j), (\varphi_i^j), 1\leq i\leq p$ be decreasing nets of FS functions converging to $\psi$ and $\varphi_i$ respectively. Then $J(\psi^j)\leq B$, and $J(\varphi_i^j)\leq B$ for $1\leq i\leq n-1$. Fix an index $j_0$, and for $j\geq j_0$ we can apply Lemma~\ref{chap4-lem-finite-on-E1} to get
\begin{align*}
    \int_{X^{\an}} \psi^{j_0}\ma(\varphi_1^j, \cdots, \varphi_p^j, \varphi_{p+1}, \cdots, \varphi_{n-1})&\geq 
    \int_{X^{\an}} \psi^{j}\ma(\varphi_1^j, \cdots, \varphi_p^j, \varphi_{p+1}, \cdots, \varphi_{n-1})\\
    &\geq \int_{X^{\an}}\psi^j\ma(\varphi_\xi) -C_nB.
\end{align*}
Letting $j\to \infty$ we get 
\[\int_{X^{\an}} \psi^{j_0}\ma(\varphi_1^j, \cdots, \varphi_p^j, \varphi_{p+1}, \cdots, \varphi_{n-1})\geq \int_{X^{\an}}\psi^{j_0}\ma(\varphi_\xi)-C_nB>-\infty.\]
Now by~\cite[Proposition 7.12]{folland} we have
\begin{align*}
    &\int_{X^{\an}} \psi\ma(\varphi_1^j, \cdots, \varphi_p^j, \varphi_{p+1}, \cdots, \varphi_{n-1}) \\
    &= \lim_{j_0\to \infty}\int_{X^{\an}} \psi^{j_0}\ma(\varphi_1^j, \cdots, \varphi_p^j, \varphi_{p+1}, \cdots, \varphi_{n-1})>-\infty.
\end{align*}

We now prove (3). Let $(\psi^j), (\varphi_i^j), 1\leq i\leq p$ be decreasing nets in $\cE^1(\xi)$ converging to $\psi$ and $\varphi_i$ respectively. Set $\mu^j = \ma(\varphi_1^j, \cdots, \varphi_p^j, \varphi_{p+1},\cdots, \varphi_{n-1})$. Then $(\mu^j)$ is a net of Radon probability measures converging weakly to $\mu\coloneqq \ma(\varphi_1, \cdots, \varphi_{n-1})$. Thus by~\cite[Corollary 2.25]{BFJ15}, we have
\[\limsup_j \int_{X^{\an}} \psi^j\mu^j\leq \int_{X^{\an}}\psi \mu.\]
The reverse direction is given by Lemma~\ref{chap4-lem-cts-dec-nets} below. Indeed, the lemma shows that for each $j$,
\begin{align*}
    &\int_{X^{\an}}\psi\mu^j\geq \int_{X^{\an}}\psi\mu^j\\
    &\geq \int_{X^{\an}}\psi \mu+\sum_{i=1}^p\int_{X^{\an}}(\varphi_i-\varphi_i^j)\ma(\varphi_1, \cdots, \varphi_{i-1}, \varphi_\xi,  \varphi_{i+1}^j, \cdots,\varphi_p^j, \varphi_{p+1}, \cdots, \varphi_{n-1})
\end{align*}
By the inductive hypothesis, the sum on the right hand side goes to $0$ as $j\to\infty$. Thus
\[\liminf_j \int_{X^{\an}}\psi^j\mu^j\geq \int_{X^{\an}}\psi\mu.\]
This completes the proof of the theorem.
\end{proof}
\begin{lem}\label{chap4-lem-cts-dec-nets}
    Let $\psi, \varphi_i'\geq \varphi_i, i=1, \cdots, p$ be functions in $\cE^1(\xi)$. Then
    \begin{align*}
        &\int_{X^{\an}} \psi\ma(\varphi_1', \cdots, \varphi_p', \varphi_{p+1},\cdots, \varphi_{n-1})\\
        &\geq \int_{X^{\an}} \psi\ma(\varphi_1, \cdots, \varphi_p, \varphi_{p+1},\cdots, \varphi_{n-1})\\
        &\indent+\sum_{i=1}^p \int_{X^{\an}}(\varphi_i-\varphi_i')\ma(\varphi_1, \cdots, \varphi_{i-1}, \varphi_\xi, \varphi_{i+1}', \cdots, \varphi_p', \varphi_{p+1}, \cdots, \varphi_{n-1}).
    \end{align*}
\end{lem}
\begin{proof}
    We may assume $\psi\in \cH(\xi)$ since Monge--Amp\`ere measures are Radon probability measures. Since these measures are also continuous along decreasing nets in the first $p$ arguments, we may in addition assume that all $\varphi_i, \varphi_i'$ are FS functions. Then integration by parts gives 
    \begin{align*}
        &\int_{X^{\an}}\psi\ma(\varphi_1', \cdots, \varphi_p', \varphi_{p+1},\cdots, \varphi_{n-1}) - \int_{X^{\an}}\psi\ma(\varphi_1,\varphi_2', \cdots, \varphi_p', \varphi_{p+1},\cdots, \varphi_{n-1})\\
        &=\int_{X^{\an}}(\varphi_1'-\varphi_1)\ma(\psi,\varphi_2', \cdots, \varphi_p', \varphi_{p+1},\cdots, \varphi_{n-1}) \\
        &\indent - \int_{X^{\an}}(\varphi_1'-\varphi_1)\ma(\varphi_\xi, \varphi_2',\cdots, \varphi_p', \varphi_{p+1},\cdots, \varphi_{n-1}).\\
    \end{align*}
    Thus we have
    \begin{align*}
        &\int_{X^{\an}}\psi\ma(\varphi_1', \cdots, \varphi_p', \varphi_{p+1},\cdots, \varphi_{n-1}) \geq \int_{X^{\an}}\psi\ma(\varphi_1,\varphi_2', \cdots, \varphi_p', \varphi_{p+1},\cdots, \varphi_{n-1})\\
        & +\int_{X^{\an}}(\varphi_1-\varphi_1')\ma(\varphi_\xi,\varphi_2', \cdots, \varphi_p', \varphi_{p+1},\cdots, \varphi_{n-1}).
    \end{align*}
    Iterating this argument for $\int_{X^{\an}}\psi\ma(\varphi_1,\varphi_2', \cdots, \varphi_p', \varphi_{p+1},\cdots, \varphi_{n-1})$ we get the desired result.
\end{proof}

The following is an immediate corollary of the previous theorem.
\begin{cor}
    Given $\varphi, \psi\in \cE^1(\xi)$, we have
    \begin{enumerate}
        \item The Monge--Amp\`ere energy of $\varphi$ satisfies
        \[\E(\varphi) = \frac{1}{n}\sum_{i=0}^{n-1}\int_{X^{\an}} \varphi \ma(\varphi^{[i]}; \varphi_\xi^{[n-1-i]}).\]
        Here $\ma(\varphi^{[i]}; \varphi_\xi^{[n-1-i]}) = \ma(\varphi, \cdots, \varphi, \varphi_\xi, \cdots, \varphi_\xi)$, where there are $i$ copies of $\varphi$, and $n-1-i$ copies of $\varphi_\xi$;
        \item $\frac{d}{dt}|_{t=0}E(t\varphi+(1-t)\psi)=\int (\varphi-\psi)\ma(\psi)$;
        \item $\E$ is homogeneous in $\xi$ of degree $-1$;
        \item When $\xi$ is rational, the finite energy class and the Monge--Amp\`ere energy recover the ones in~\cite[Section 6]{BJ18v1} and~\cite[Section 7]{BJ18}.
    \end{enumerate}
\end{cor}

\subsection{$\I$ and $\J$ functionals on $\cE^1$}\label{subsec: quasimetric on E^1}
To simplify notation, we will write $\ma(\varphi) = \ma(\varphi_1, \cdots, \varphi_{n-1})$ provided $\varphi = \varphi_1=\cdots =\varphi_{n-1}$.
\begin{defn}
    For $\varphi, \psi \in \cE^1(\xi)$, set
        \[\I(\varphi, \psi)\coloneqq \int_{X^{\an}} (\varphi-\psi)(\ma(\psi)-\ma(\varphi)),\]
    and 
    \[\J_{\psi}(\varphi)\coloneqq \int_{X^{\an}} (\varphi-\psi) \ma(\psi) - \E(\varphi, \psi).\]
    We will simply write $\I(\varphi), \J(\varphi)$ if $\psi = \varphi_\xi$.
\end{defn}

As is for continuous psh functions, the $\I$ and $\J$ functionals satisfy the same properties on $\cE^1(\xi)$, which we summarize in the following proposition. In particular, the $\I$ functional defines a quasi-pseudometric on $\cE^1(\xi)$.
\begin{prop}\label{prop-IJ-E1}
     For $\varphi, \varphi', \psi\in \cE^1(\xi)$, we have
    \begin{enumerate}
        \item $\I(\varphi, \psi)=\vol(\xi)^{-1} \sum_{j=0}^{n-2} \|\varphi-\psi\|^2_{(\varphi^j, \psi^{n-2-j})};$
        \item $\J_\psi(\varphi) = \vol(\xi)^{-1} \sum_{j=0}^{n-2}\frac{j+1}{n} \|\varphi-\psi\|^2_{(\varphi^j, \psi^{n-2-j})};$
        \item $\frac{1}{n} \I(\varphi, \psi)\leq \J_\psi(\varphi)\leq \frac{n-1}{n} \I(\varphi,\psi).$
        \item $\I(\varphi, \varphi')\lesssim \I(\varphi, \psi)+\I(\psi, \varphi')$;
        \item Set $\alpha_n = \frac{1}{2^{n-1}}$. We have
    \[\left|\int (\varphi-\varphi')(\ma(\psi)-\ma(\psi'))\right|\lesssim \I(\varphi, \varphi')^{\alpha_n}\I(\psi, \psi')^{\frac 12}\max\{\J(\varphi), \J(\varphi'), \J(\psi), \J(\psi')\}^{\frac 12-\alpha_n}.\]
    \end{enumerate}
\end{prop}

\begin{cor}
    For $\varphi,\varphi', \psi\in\cE^1(\xi)$, we have
    \[\left|\J_\psi(\varphi)-\J_\psi(\varphi')\right|\lesssim \I(\varphi, \varphi')^{\alpha_n}\max\{\J(\varphi), \J(\varphi'), \J(\psi)\}^{1-\alpha_n},\]
    where $\alpha_n = \frac{1}{2^{n-1}}$.
\end{cor}
\begin{proof}
    Direct computation shows that
    \[\J_\psi(\varphi)-\J_\psi(\varphi')= \int (\varphi-\varphi')\ma(\psi)-\E(\varphi)+\E(\varphi') = \J_{\varphi'}(\varphi)+\int(\varphi-\varphi')(\ma(\psi)-\ma(\varphi')).\]
    Hence Proposition~\ref{prop-IJ-E1} yields
    \begin{align*}
        \left|\J_\psi(\varphi)-\J_\psi(\varphi')\right|&\lesssim \I(\varphi, \varphi')+\I(\varphi, \varphi')^{\alpha_n}\I(\psi, \varphi')^{\frac 12}\max\{\J(\varphi), \J(\varphi'), \J(\psi)\}^{\frac 12-\alpha_n}\\
        &\lesssim  \I(\varphi,\varphi')^{\alpha_n} [(\I(\varphi)+\I(\varphi'))^{1-\alpha_n}+ (\I(\psi)+\I(\varphi'))^{\frac 12}\max\{\J(\varphi), \J(\varphi'), \J(\psi)\}^{\frac 12-\alpha_n} ]\\
        &\lesssim \I(\varphi, \varphi')^{\alpha_n}\max\{\J(\varphi), \J(\varphi'), \J(\psi)\}^{1-\alpha_n}.
    \end{align*}

\end{proof}

This also yields the following concavity of the Monge--Amp\`ere energy.
\begin{prop}\label{prop:E-strict-concavity}
    For $\varphi, \psi\in \cE^1$ and $t\in [0,1]$, we have
    \[\E(t\varphi+(1-t)(\psi))-t\E(\varphi)-(1-t)\E(\psi)\gtrsim t(1-t)\I(\varphi, \psi).\]
\end{prop}
\begin{proof}
    Set $\varphi_t\coloneqq t\varphi+(1-t)\psi$. By the Proposition~\ref{prop-IJ-E1}, we have
    \[\frac 1n \I(\varphi, \varphi_t)\leq \int(\varphi-\varphi_t)\ma(\varphi_t)-\E(\varphi)+\E(\varphi_t),\]
    and similarly
    \[\frac 1n \I(\varphi_t, \psi)\leq \int(\psi-\varphi_t)\ma(\psi)+\E(\varphi_t)-\E(\psi).\]
    Combining these together, with~\cite[Lemma 7.29]{BJ18}, we have
    \[\E(\varphi_t)-t\E(\varphi)-(1-t)\E(\psi)\gtrsim t\I(\varphi, \varphi_t)+(1-t)\I(\varphi_t, \psi)\geq t(1-t)(\I(\varphi, \varphi_t)+\I(\varphi_t, \psi))\gtrsim t(1-t)\I(\varphi, \psi),\]
    where the last inequality again follows from Proposition~\ref{prop-IJ-E1}.
\end{proof}


\section{Measures of finite energy}\label{section: FiniteEnergyMeasure}
Finally, we begin the study of measures of finite energy. 
Fix a polarized affine cone $(X; \xi)$. Denote by $\cM = \cM(X_0)$ the space of Radon probability measures on $X_0$.
\begin{defn}
    The energy of $\mu\in \cM$ with respect to the polarization $\xi$ is defined by 
    \[\E^\vee(\mu;\xi)\coloneqq \sup_{\varphi\in \cE^1(\xi)}\left(\E(\varphi)-\int_{X_0}\varphi\mu\right)\in \RR\cup \{+\infty\}.\]
    A measure $\mu\in \cM$ is said to have \emph{finite energy} if $\E^\vee(\mu;\xi)<+\infty$. Denote by $\cM^1(\xi)\subset \cM$ the space of measures of finite energy with polarization $\xi$. 
\end{defn}
By definition, if $\mu$ has finite energy, then $\int \varphi\mu$ is finite for any $\varphi\in \cE^1(\xi)$. We also note that if $\xi$ is rational and primitive, then the above definition agrees with the one in~\cite[Definition 9.1]{BJ18}. In general, it is homogeneous in $\xi$ of degree $-1$. Alternatively, following the strategy in~\cite{BJNASyn}, one can define measures of finite energy using only continuous psh functions, and the results in this section will remain valid on $\cpsh(\xi)$.
\begin{lem}\label{chap4-lem-sup-over-fs}
The $\sup$ in the above definition can also be taken over FS functions:
    \[\E^\vee(\mu;\xi)= \sup_{\varphi\in \cH(\xi)}\left(\E(\varphi)-\int_{X_0}\varphi\mu\right).\]
\end{lem}
\begin{proof}
    The direction $\geq$ is clear. For any given $\varphi\in \cE^1(\xi)$, pick a decreasing net $\varphi_j\to \varphi$ in $\cH(\xi)$. Then Theorem~\ref{chap4-thm-mixed-ma} implies that 
    \[\E(\varphi_j)-\int\varphi_j\mu\geq \left(\E(\varphi)-\int\varphi\mu\right) + \int (\varphi-\varphi_j)\mu\]
    for all $j$. We are done by letting $j\to \infty$ and the monotone convergence theorem (see e.g.~\cite[Lemma 7.17]{BJ18}).
\end{proof}
\begin{prop}\label{chap4-prop-E-duality}
    For any $\varphi\in \cE^1(\xi)$, 
    \[\E(\varphi)= \inf_{\mu\in \cM^1(\xi)}\left(\E^\vee(\mu)+\int \varphi\mu\right).\]
\end{prop}
\begin{proof}
    By definition, if $\mu \in \cM^1(\xi)$, then
    \[\E(\varphi)\leq \E^\vee(\mu)+ \int \varphi\mu.\]
    For the other direction, we first claim that for any $\psi\in \cE^1(\xi)$, 
    \[\E(\varphi)-\E(\psi)\leq \int (\varphi-\psi) \ma(\varphi).\]
    Indeed, if $\varphi, \psi$ are both FS functions, then this is a direct consequence of (2) and (4) in Proposition~\ref{chap4-prop-ma-properties}. In general, $\varphi, \psi$ are decreasing limits of FS functions, and the claim follows by Theorem~\ref{chap4-thm-mixed-ma}.
    With the claim, we have
    \[\E^\vee(\ma(\varphi); \xi) = \sup_{\psi\in \cE^1(\xi))}\left(\E(\psi)-\int\psi\ma(\varphi)\right) = \E(\varphi)-\int\varphi\ma(\varphi).\]
    This proves the equality.
\end{proof}

\begin{prop}\label{chap4-prop-energy-S-duality}
    Let $v$ be a $\TT$-invariant quasi-monomial valuation in $X_0$. Then $v\in \cM^1(\xi)$, and 
    \[\E^\vee(\delta_v;\xi)\leq \frac{n+1}{n}S(v;\xi),\]
    where $\delta_v$ is the dirac mass at $v$.
\end{prop}

\begin{proof}
    We have 
    \begin{align*}
        \E^\vee(\delta_v;\xi) &= \sup_{\varphi\in \cE^1(\xi)}\left(\E(\varphi)-\varphi(v)\right)\leq \sup_{\varphi\in \cE^1(\xi)}\left(\sup_{X_0}\varphi-\varphi(v)\right)\\
        &\leq\sup_{\varphi\in \psh(\xi)}\left(\sup_{X_0}\varphi-\varphi(v)\right)=T(v,\xi),
    \end{align*}
    where the last equality follows from Corollary~\ref{chap4-cor-T-psh-duality}. Since $v$ has linear growth as shown in Corollary~\ref{chap4-cor-qm-linear-growth}, we have $v\in \cM^1(\xi)$.
    To prove the second assertion, pick $\xi_j\to \xi$ with $\xi_j$ rational. Then~\cite[Theorem 7.22]{BJNA1} in our notation reads
    \[E^\vee(\delta_v;\xi_j) = \frac{n+1}{n}S(v;\xi_j)\]
    for all $j$. By Lemma~\ref{chap4-lem-sup-over-fs}, $E^\vee(\delta_v; \cdot)$ is lower-semicontinuous in $\xi$. So the proof is complete by letting $j\to \infty$.
\end{proof}

\subsection{$\I$ and $\J$ functionals on $\cM^1$}
In a similar manner as in Section~\ref{subsec: quasimetric on E^1}, one can dually define $\I^\vee$ and $\J^\vee$ functionals on $\cM^1(\xi)$. The duality characterization below will finish the proof of Theorem~\ref{main-thm-pluripotential}.
\begin{defn}
    For any $\mu\in\cM^1(\xi)$, we define $\J_\mu: \cE^1(\xi)\to [0, \infty)$ by
    \[\J_\mu(\varphi)\coloneqq \E^\vee(\mu)-\E(\varphi)+\int \varphi\mu, \]
    and for any two $\mu, \mu'\in \cM^1(\xi)$ we define
    \[\I^\vee(\mu, \mu')\coloneqq \inf_{\varphi\in \cE^1}(\J_\mu(\varphi)+\J_{\mu'}(\varphi)).\]
\end{defn}

\begin{lem}
    For all $\varphi, \varphi'\in \cE^1(\xi)$, we have
    \[\I(\varphi, \varphi')\approx \inf_{\mu\in\cM^1}(\J_\mu(\varphi)+\J_\mu(\varphi').\]
    In particular, for all $\varphi\in \cE^1$ and $\mu\in \cM^1(\xi)$,
    \[\J(\varphi)\lesssim \J_\mu(\varphi)+\E^\vee(\mu).\]
\end{lem}
\begin{proof}
    First note that 
    \[\inf_{\mu\in\cM^1(\xi)}(\J_\mu(\varphi)+\J_\mu(\varphi'))\leq \J_{\ma(\varphi)}(\varphi)+\J_{\ma(\varphi)}(\varphi')=\J_\varphi(\varphi')\approx \I(\varphi, \varphi'),\]
    by Proposition~\ref{prop-IJ-E1}. On the other hand, for $\tau=\frac{\varphi+\varphi'}{2}$, by Proposition~\ref{prop:E-strict-concavity} we have
    \[\E^\vee(\mu)\geq \E(\tau)-\int \tau\mu\geq \frac 12(\E(\varphi)+\E(\varphi'))-\frac 12\int(\varphi+\varphi')\mu+C(n)\I(\varphi, \varphi')\]
    for some dimensional constant $C(n)$. Hence
    \[\inf_{\mu\in\cM^1}(\J_\mu(\varphi)+\J_\mu(\varphi'))\gtrsim \I(\varphi, \varphi').\]
    The second claim follows by taking $\varphi'=\varphi_\xi$.
\end{proof}

\begin{lem}
    For $\varphi, \varphi'\in \cE^1(\xi)$ and $\mu=\ma(\varphi), \mu'=\ma(\varphi')$, we have
    \[\I(\varphi,\varphi')\approx \I^\vee(\mu, \mu').\]
\end{lem}
\begin{proof}
    This follows from unwinding definitions and Proposition~\ref{prop-IJ-E1}:
    \[\I^\vee(\mu, \mu')=\inf_{\psi\in \cE^1} (\J_\varphi(\psi)+\J_{\varphi'}(\psi))\approx \inf_{\psi\in \cE^1} (\I(\varphi, \psi)+\I(\varphi', \psi))\approx \I(\varphi, \varphi'). \]
\end{proof}

As a direct consequence, we have
\begin{cor}
     The Monge--Amp\`ere operator defines a bi-Lipschitz map 
     \[\ma: (\cE^1(\xi), \I)\to (\cM^1(\xi), \I^\vee).\]
\end{cor}

\bibliographystyle{alpha}
\bibliography{bibliography}

\end{document}